\documentclass{amsart}
\usepackage{graphicx}
\usepackage{fullpage}
\usepackage{color}
\usepackage{comment}
\usepackage{paralist}
\usepackage{hyperref}

\newtheorem{thm}{Theorem}[section]
\newtheorem{prop}[thm]{Proposition}
\newtheorem{lem}[thm]{Lemma}
\newtheorem{cor}[thm]{Corollary}

\newtheorem{asm}{Assumption}

\newtheorem{remark}{Remark}

\newcommand{\ra}{\rightarrow}

\newcommand{\Z}{\mathbb Z}     

\renewcommand{\a}{\alpha}
\renewcommand{\b}{\beta}
\renewcommand{\d}{\delta}

\renewcommand{\l}{\lambda}

\newcommand{\w}{\omega}              

\renewcommand{\P}{\mathbb{P}_\a}        
\newcommand{\E}{\mathbb{E}_\a}          

\newcommand{\fl}[1]{\lfloor #1 \rfloor}  

\author{Sung Won Ahn}
\address{Sung Won Ahn\\Purdue University\\Department of Mathematics\\150 N University Street\\West Lafayette, IN  47907\\USA}
\email{asw0205@gmail.com}

\author{Jonathon Peterson}
\address{Jonathon Peterson\\Purdue University\\Department of Mathematics\\150 N University Street\\West Lafayette, IN  47907\\USA}
\email{peterson@purdue.edu}
\urladdr{www.math.purdue.edu/~peterson}
\thanks{J. Peterson was partially supported by NSA grants H98230-13-1-0266 and H98230-15-1-0049.}

\begin{document}

\title[Oscillations of quenched slowdown asymptotics]{Oscillations of quenched slowdown asymptotics for ballistic one-dimensional random walk in a random environment}

\begin{abstract}
We consider a one dimensional random walk in a random environment (RWRE) with a positive speed $\lim_{n\to\infty}\frac{X_n}{n}=v_\alpha>0$.
Gantert and Zeitouni \cite{Zei98} showed that if the environment has both positive and negative local drifts then the quenched slowdown probabilities $P_\omega(X_n < xn)$ with $x \in (0,v_\alpha)$ decay approximately like $\exp\{-n^{1-1/s}\}$ for a deterministic $s > 1$.
More precisely, they showed that $n^{-\gamma} \log P_\omega( X_n < x n)$ converges to $0$ or $-\infty$ depending on whether $\gamma > 1-1/s$ or $\gamma < 1-1/s$.
In this paper, we improve on this by showing that $n^{-1+1/s} \log P_\omega( X_n < x n)$ oscillates between $0$ and $-\infty$, almost surely.
This had previously been shown only in a very special case of random environments \cite{Gan02}.
\end{abstract}
\maketitle
\section{Introduction}
Let $\omega=\{\omega_z\}\in[0,1]^{\mathbb{Z}}$ be a sequence of independent, identically distributed random variables called an environment, and let $\alpha$ be the distribution of $\w$ on the space $[0,1]^\Z$ of all environments.  For a given environment $\omega$, we can generate a random path, $X_n, n\in\mathbb{N}$, with transition probability
\begin{align*}
&P_\omega(X_{n+1}=x+1|X_n=x)=\omega_x\\
&P_\omega(X_{n+1}=x-1|X_n=x)=1-\omega_x.
\end{align*}
The process generated in this way is called a random walk in a random environment (RWRE).
If the path $X_n$ starting at $x$ is generated under one particular environment $\omega$, the corresponding law is called quenched law denoted by $P_\omega^x(\cdot)$, and its expectation is denoted by $E_\omega^x[\cdot]$.  Without conditioning on the environment $\omega$, the law of $X_n$ starting at $x$ is called the annealed law denoted by
$\P^x(\cdot) = E_\a[ P_\w^x(\cdot) ]$, where $E_\a[\cdot]$ denotes expectation with respect to the measure $\a$ on environments.
Expectations under the annealed measure will be denoted by $\E^x[\cdot]$.  For simplicity we write $P_\w(\cdot)$, $E_\omega[\cdot]$, $\P(\cdot)$, $\E[\cdot]$ when the walk is started at $x=0$.

The first mathematical result for RWRE was the limit behaviors of $X_n$ by Solomon in \cite{Sol75}.  Solomon proved that the recurrence or transience of the RWRE is characterized by the sign of $E_\a[\log\rho_0]$, where the random variables $\rho_x$ are defined by $\rho_x= (1-\omega_x)/\omega_x$.  In Solomon's paper, he showed that the RWRE is transient to $+\infty$ if $E_\a[\log\rho_0]<0$, transient to $-\infty$ if $E_\a[\log\rho_0]>0$, and recurrent if $E_\a[\log\rho_0]=0$.  Further, he also proved that $X_n$ satisfies a law of large numbers, developing an explicit formula of the speed of RWRE.
In particular, Solomon showed that the limit $v_\alpha=\lim_{n\to\infty}\frac{X_n}{n}$ exists $\P$-a.s., and if the walk is transient to the right (i.e., $E_\alpha[\log \rho_0]<0$) the speed $v_\alpha$ is given by
\begin{equation}\label{con1}
 v_\alpha =
\begin{cases}
 \frac{1-E_\a[\rho_0]}{1+E_\a[\rho_0]} & \text{if } E_\a[\rho_0]<1 \\
 0 & \text{if } E_\a[\rho_0]\geq 1.
\end{cases}
\end{equation}
An extension to Solomon's work, the limiting distributions of transient RWRE under the annealed law, was studied by Kesten, Kozlov, and Spitzer.  In their paper, a parameter $s>0$, defined by the equation
$$E_\a[\rho^s_0]=1,\quad s>0,$$
proved to be a key factor determining both the scaling factor and the limit law of the random walk.   In part,
\begin{itemize}
\item If $s\in(1,2)$, then under annealed law, $\frac{X_n-nv_\alpha}{n^{1/s}}\Rightarrow$ a stable law of index $s$.
\item If $s>2$, then under annealed law with a constant $\sigma>0$, $\frac{X_n-nv_\alpha}{\sigma\sqrt{n}}\Rightarrow$ a standard normal law.
\end{itemize}
(Here, and throughout the paper, we will use $\Rightarrow$ to denote convergence in distribution.)  Limiting distributions for $s \in(0,1]$ and $s=2$ are also shown in \cite{KKS}.


The main results in the present paper concern large deviations of RWRE.
A large deviation principle (LDP) for $X_n/n$ under the quenched measure was first proved by Greven and den Hollander \cite{GH94}.
Later, Comets, Gantert, and Zeitouni \cite{Zei03} used a different approach, obtaining a LDP for $X_n/n$ as a byproduct of a LDP for $T_n/n$, where $T_n:=\inf\{i\geq 0:X_i=n\}$ is the hitting time of site $n$.
This approach had the advantage of giving LDPs under both the quenched and annealed measures.
Moreover, the approach in \cite{Zei03} led to a good qualitative description of the quenched and annealed large deviation rate functions.
Our interest in the present paper concerns certain large deviation asymptotics when the RWRE is positive speed and with mixed local drifts;
that is, $v_\alpha > 0$ and $\alpha(\omega_0 \leq 1/2) > 0$.
In this case, the results in \cite{Zei03} show that both the quenched and averaged large deviation rate functions vanish on the interval $[0,v_\alpha]$. That is,
\[
\lim_{n\to\infty}\frac{1}{n}\log P_\omega\left(\frac{X_n}{n}<v \right)=\lim_{n\to\infty}\frac{1}{n}\log \P\left(\frac{X_n}{n}<v \right)=0,\quad{}v\in[0,v_\alpha].
\]
Thus, in the case of positive speed with mixed local drifts, the probability of the random walk moving at a positive but slower than typical speed decays sub-exponentially in $n$.
It was shown in several papers that the precise rate of decay of these large deviation slowdown probabilities is different under the quenched and annealed measures and that the sub-exponential rate depends on the specifics of the distribution $\alpha$ on environments \cite{DPZ,Zei98,PPZ}.
Our interest in this paper concerns the rate of decay of the quenched probabilities $P_\omega(X_n < nv)$
under the following assumptions.
\begin{asm}\label{asm:s}
The distribution $\alpha$ on environments is such that $E_\a[ \log \rho_0 ] < 0$ and $E_\a[ \rho_0^s ] = 1$ for some $s > 1$.
\end{asm}
\begin{remark}
 It follows from H\"older's inequality that $\gamma \mapsto E_\a[ \rho_0^\gamma]$ is a convex function. Moreover, the slope of this function at $\gamma = 0$ is $E_\a[\log \rho_0] < 0$ and thus it follows from Assumption \ref{asm:s} that $\E[\rho_0] < 1$ and therefore the RWRE is transient to the right with positive speed $v_\alpha > 0$. Moreover, since $\rho_0 < 1 \iff \omega_0 > 1/2$ it follows that $\alpha(\omega_0 > 1/2)>0$ and $\alpha(\omega_0 < 1/2) > 0$. Since the environment is assumed to be i.i.d.\ this implies that $\alpha$-a.e.\ environment has sites with local drifts to the right and to the left.
\end{remark}

In addition to Assumption \ref{asm:s}, we will also need the following technical assumptions.
\begin{asm}\label{asm:s2}
The distribution of $\log\rho_0$ is non-lattice under $\alpha$ and that $E_\alpha[ \rho_0^s \log \rho_0 ] < \infty$.
\end{asm}
\begin{remark}
The conditions in Assumption \ref{asm:s2} are needed for certain precise tail asymptotics that we will use throughout the paper. It may be that the main results of this paper are true without these additional technical assumptions, but this would require dealing with rougher tail asymptotics throughout the paper. The conditions in Assumption \ref{asm:s2} have also been used in many previous papers in one-dimensional RWRE \cite{KKS},\cite{GS02},\cite{PZ09},\cite{FGP10},\cite{PS13},\cite{ESTZ13},\cite{DG12}.
\end{remark}

The asymptotics of the quenched slowdown probabilities under Assumption \ref{asm:s} were first studied by Gantert and Zeitouni in \cite{Zei98}.
In particular, Gantert and Zeitouni proved that for any $v \in (0,v_\alpha)$ and any $\d>0$,
\begin{align*}
    &\lim_{n\to\infty}\frac{1}{n^{1-1/s+\delta}}\log P_\omega\left(\frac{X_n}{n}\leq v\right) =  0, \qquad \a\text{-a.s.}\\
    &\lim_{n\to\infty}\frac{1}{n^{1-1/s-\delta}}\log P_\omega\left(\frac{X_n}{n}\leq v\right) =  -\infty, \qquad \a\text{-a.s.}
\end{align*}
One might suspect from this that $P_\omega(X_n/n\leq v)$ decays on a stretched exponential scale like $\exp(-Cn^{1-1/s})$ for some deterministic constant $C>0$ depending on $v\in(0,v_\alpha)$.  However, in \cite{Zei98}
Gantert and Zeitouni showed that for any $v \in (0,v_\a)$,
\begin{equation}
\limsup_{n\to\infty}\frac{1}{n^{1-1/s}}\log P_\omega(\frac{X_n}{n}<v)=0, \qquad \a\text{-a.s.,}\label{eq0}
\end{equation}
and conjectured that the corresponding $\liminf$ is equal to $-\infty$.
The main result of our paper complements \eqref{eq0} by proving this conjecture.
\begin{thm}\label{thm0}
If Assumptions \ref{asm:s} and \ref{asm:s2} hold, then for any $v \in (0,v_\alpha)$,
\begin{equation}
\liminf_{n\to\infty}\frac{1}{n^{1-1/s}}\log P_\omega\left(\frac{X_n}{n}<v \right)=-\infty, \qquad \alpha-a.s.\label{thm00}
\end{equation}
\end{thm}
Together with \eqref{eq0}, we conclude that $\frac{1}{n^{1-1/s}}\log P_\omega(X_n/x<v)$ fluctuates between $0$ and $-\infty$, $\alpha$-a.s.
\begin{remark}
\begin{enumerate}[(i)]
\item Theorem \ref{thm0} was proved in a special case by Gantert in \cite{Gan02} in which $\alpha(\omega_0  \in \{p,1\}) = 1$ for some fixed $p<1/2$. In this case, the environment $\omega$ consists of scattered ``one-way nodes'' (i.e., sites $x$ with $\omega_x = 1$) and all remaining sites have a fixed drift to the left.
We note that the results in \cite{Gan02} also include cases where the distribution $\alpha$ is such that the environment $\omega=\{\omega_x\}_{x\in\mathbb{Z}}$ is ergodic rather than i.i.d.
In the present paper we restrict ourselves to only i.i.d.\ environments but remove the requirement that the support of $\omega_0$ is $\{p,1\}$.
\item In the same setting of Theorem \ref{thm0}, it was shown in \cite{Zei98} that the corresponding annealed probabilities decay polynomially fast. In particular,
    \begin{equation*}
    \lim_{n\to\infty}\frac{1}{\log n}\log \P(X_n < nv)=1-s, \quad \forall v \in (0,v_\alpha).
    \end{equation*}
    Observe that the decay rate of the annealed case is slower than that of the quenched case due to the extra randomness available in choosing an environment $\omega$ from $\Omega$.
\item The precise sub-exponential quenched and annealed rates of decay of the slowdown probabilities has also been studied under the assumption that the environment has ``positive or zero drift;'' that is, $\alpha( \omega_0 \geq 1/2 ) = 1$ and $\alpha_0 := \alpha(\omega_0 = 1/2) \in (0,1)$.
In this case the precise quenched and annealed asymptotics of the slowdown probabilities were given in \cite{PP} an \cite{PPZ}, respectively.
In particular,
\begin{equation}\label{eq:qpzd}
 \lim_{n\to\infty}\frac{(\log n)^2}{n}\log P_\omega(X_n<nv)= - \frac{( \pi \log \alpha_0)^2}{8}\left(1 - \frac{v}{v_\alpha} \right), \quad \forall v \in (0,v_\alpha),
\end{equation}
and
\[
 \lim_{n\to\infty}\frac{1}{n^{1/3}}\log \P(X_n<nv)= -\left\{ \frac{27 (\pi \log \alpha_0)^2}{32}\left(1 - \frac{v}{v_\alpha} \right) \right\}^{1/3} , \quad \forall v \in (0,v_\alpha).
\]
In particular, note that the existence of the quenched limit in \eqref{eq:qpzd} contrasts with Theorem \ref{thm0} and \eqref{eq0}.
\end{enumerate}
\end{remark}

\subsection{Notation and Background}
Before beginning the proof of Theorem \ref{thm0}, we introduce some notation that will be used throughout the remainder of the paper.
First, we note that throughout paper, we will use $c,\,c',\,C,\,C',...$ as a generic positive constants whose values are not important and may differ by one usage to another, and use $C_0,\,C_1,\,C_2,...$ as constants constructed for a specific usage.

Recall that for an environment $\w = (\w_x)_{x\in\Z}$, we have defined $\rho_x = \frac{1-\w_x}{\w_x}$. Then, for any integers $i\leq j$ we define
$$\Pi_{i,j}:=\prod_{k=i}^j\rho_k,\quad W_{i,j}:=\sum_{k=i}^j\Pi_{k,j},\quad R_{i,j}:=\sum_{k=i}^j\Pi_{i,k}
$$
$$W_j:=\sum_{k\leq j}\Pi_{k,j},\quad R_i:=\sum_{k=i}^\infty\Pi_{i,k}.
$$
(Note that $W_i$ and $R_i$ are finite
for all $i\in\mathbb{Z}$ with probability one if $E_\a[\log \rho_0] < 0$.)
We will use these notations frequently in the next sections in order to simplify various expressions under the quenched law. In particular, note that we can obtain a quenched expectation of $\tau_i=T_{i+1}-T_i$ (the time to cross from $i$ to $i+1$) by
\begin{equation}
E_\omega[\tau_i]=1+2W_i,\label{beta}
\end{equation}
which is derived from \cite[(2.1.7) and (2.1.8)]{SZ04}.

Throughout this paper, we will use the method introduced by Sinai of the ``potential'' of an environment which allows us to visualize the environment as a sequence of ``valleys''\cite{Sin84}.
This technique was originally developed by Sinai to study the limiting distributions of recurrent RWRE but has also shown to be useful for transient RWRE \cite{PZ09},\cite{FGP10},\cite{PS13},\cite{ESTZ13}. For a fixed environment $\omega$, let the potential $V(x)$ be the function

$$V(x) =
  \begin{cases}\
    \sum_{i=0}^{x-1}\log\rho_i & \quad \text{if $x\geq1$}\\
    0 & \quad \text{if $x=0$}\\
    -\sum_{i=x}^{-1}\log\rho_i & \quad \text{if $x\leq-1$}.
  \end{cases}$$
The potential $V(x)$ enables us to cut an environment into blocks by ``ladder points'', $\{\nu_i, i\in\mathbb{Z}\},$ defined by
\begin{equation}
\nu_{0}=\sup\{y\leq 0:V(y)<V(k),\forall k<y\},\label{nu0}
\end{equation}
and for $i\geq 1$,
$$\nu_i=\inf\{x>\nu_{i-1}:V(x)< V(\nu_{i-1})\}, \quad\text{and}\quad \nu_{-i}=\sup\{y<\nu_{-i+1}:V(y)<V(k),\forall k<y\}.$$
Equivalently,
$$\nu_0=\sup\{y\leq 0:\Pi_{k,y-1}<1,\forall k<y\}$$
and, for $i\geq 1$,
$$\nu_i=\inf\{x>\nu_{i-1}:\Pi_{\nu_{i-1},x-1}<1\},\quad \text{and}\quad\nu_{-i}=\sup\{y<\nu_{-i+1}:\Pi_{k,y-1}<1,\forall k<y\}.$$
Figure \ref{fig1} is an example of the locations of ladder points on $\mathbb{Z}$.
\begin{figure}
    \centering
    \includegraphics[scale=1]{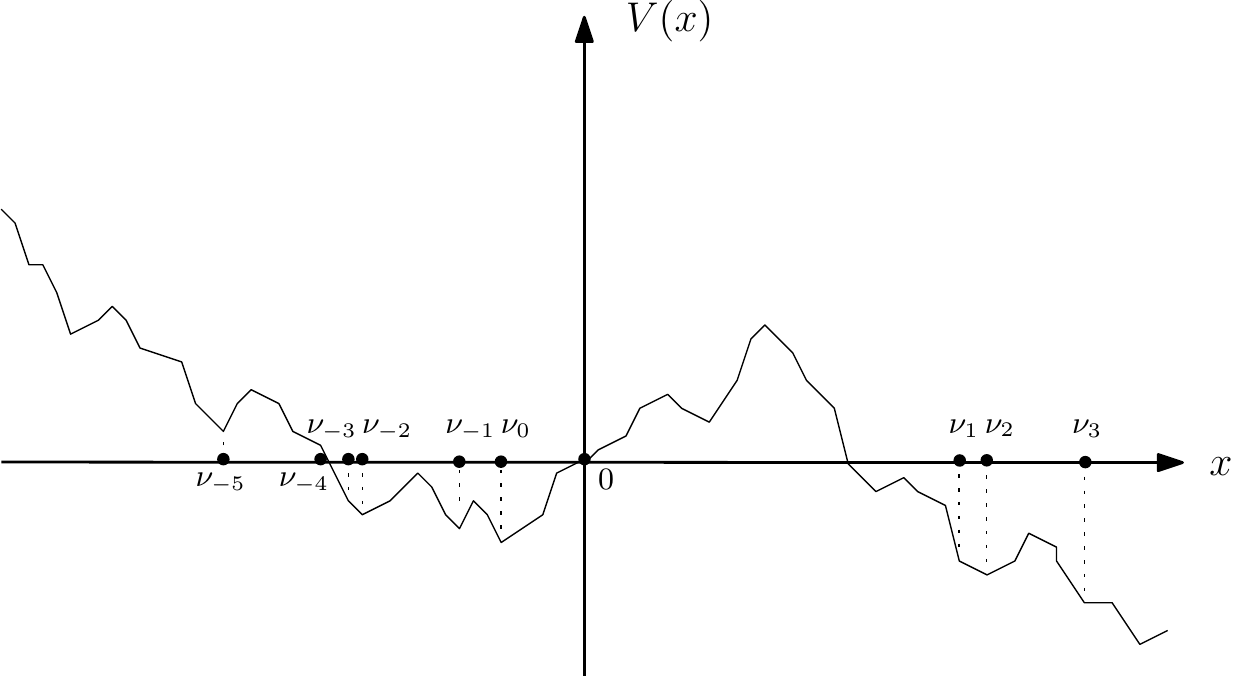}
    \caption{The locations of ladder points $\{\nu_i\}_{i\in\mathbb{Z}}$ on $\mathbb{Z}$.}
    \label{fig1}
\end{figure}
Let us denote the length between consecutive ladder points by
$$l_{i}=\nu_{i+1}-\nu_{i}, \qquad i\in\mathbb{Z},$$
and the exponential height of the potential between the ladder points by
$$M_i:=\max\{\Pi_{\nu_i,j}:\nu_i\leq j\leq \nu_{i+1}\}=\max\{e^{V(j)-V(\nu_{i})}:\nu_i<j\leq \nu_{i+1}\}, \qquad i \in \mathbb{Z}.$$
This exponential height has a crucial role in our analysis because our result shows that the quenched expectation of the crossing times on sections with ``big'' $M_i$ determines which subsequence to take for Theorem \ref{thm0} to be satisfied.  Also, we will show that the sums of the quenched expectation of crossing times on sections with a ``small'' $M_i$ is negligible in the limit.

The ladder points of the environment form a convenient structure for studying the hitting times of the random walk.
Since the environment is i.i.d.\ under the measure $\alpha$, it follows that the blocks of the environment between adjacent ladder points
$\mathfrak{B}_i = \{\omega_x: \, x \in [\nu_i,\nu_{i+1}) \}$
are i.i.d.\ for $i\neq 0$. In particular, $\{l_i\}_{i\neq 0}$ and $\{M_i\}_{i\neq 0}$ are both i.i.d.\ sequences of random variables.
However, the interval of environment between the ladder points on either side of the origin has a different distribution. In particular, under the measure $\alpha$, the random variables $l_0$ and $M_0$ have a different distribution that $l_i$ and $M_i$ with $i\neq 0$.
For this reason, it is convenient to at times work with a related measure on environments $Q$ given by
\[
 Q(\cdot) = \alpha( \cdot | \nu_0 = 0 ).
\]
The sequence $\{\omega_x\}_{x\in\mathbb{Z}}$ is no longer i.i.d.\ under the measure $Q$, but this distribution has the convenient property that the environment is stationary under shifts of the ladder points of the environment.
More precisely, if $\theta$ is the natural left-shift operator on environments given by $(\theta \omega)_x=\omega_{x+1}$, then for any $k\in \mathbb{Z}$ the environments $\omega$ and $\theta^{\nu_k}\omega$ have the same distribution under $Q$.
Moreover, under the measure $Q$ the blocks between adjacent ladder points $\mathfrak{B}_i$ are i.i.d.\ for all $i \in \mathbb{Z}$ with each having the same distribution as $\mathfrak{B}_1$ under the original measure $\alpha$ on environments.
In particular, this implies that $\{l_i\}_{i\in\mathbb{Z}}$ and $\{M_i\}_{i \in \mathbb{Z}}$ are both i.i.d.\ sequences under the measure $Q$.

The distribution $Q$ was first introduced in \cite{PZ09}, and we will frequently refer to estimates under the measure $Q$ that were proved in this paper. We mention here a few of these that we will use throughout the remainder of the paper. First of all, under the measure $Q$ the distances $l_i$ between ladder points have exponential tails. That is, there exist constants $C,C'>0$ such that
\begin{equation}\label{Qltails}
 Q(l_i > x) \leq C e^{-C'x}.
\end{equation}
Secondly, it follows from a result of Iglehart \cite[Theorem 1]{Igl72} that there exists a constant $C_0>0$ such that
\begin{equation}
Q(M_i>x)\sim C_0 x^{-s},\quad\text{ as }x\to\infty.\label{M}
\end{equation}
(Note that it follows from this asymptotic statement that $Q(M_i > x) \leq C x^{-s}$ for all $x>0$ and some $C>0$. At times we will use this upper bound rather than the asymptotics in \eqref{M}.)
One of the main ideas that will be used throughout the paper is that the expected time for the random walk to cross between adjacent ladder points $E_\omega^{\nu_i}[T_{\nu_{i+1}}]$ is roughly comparable to the exponential height of the potential $M_i$ between the ladder points. Thus, we expect that $E_\omega[T_{\nu_1}]$ also has polynomial tails similar to \eqref{M}. Indeed, it was shown in \cite{PZ09} that
\begin{equation}\label{Qbtails}
 Q(E_\omega[T_{\nu_1}]>x)\sim K_\infty x^{-s},\quad\forall x\geq0,
\end{equation}
for some $K_\infty>0$.

We conclude the introduction with an overview of the proof of Theorem \ref{thm0}.
As in \cite{Zei98}, we will study the slowdown probabilities through the hitting times of the random walk. That is, we will prove Theorem \ref{thm0} by proving that $\liminf_{n\to\infty} n^{-1+1/s} \log P_\omega( T_n > u n) = -\infty$ for all $u>1/v_\alpha$.
We will first show that this limit holds for $Q$-a.e.\ environment $\omega$ and then from this deduce that the limit also holds almost surely under the original measure $\alpha$ on environments.
The proof of the quenched slowdown asymptotics for the hitting times is structured as follows.
In Section \ref{sec:mgf}, we give an explicit upper bound of the quenched moment generating function of the hitting times with a one way node placed on a site to the left of the starting point.  This explicit form shows that the sums of quenched expected time between ladder locations control the quenched subexponential tail of hitting times.  In Section \ref{sec:ald}, we will show the sums of the quenched expected crossing time between ladder locations with ``small'' $M_i$ are negligible in the limit under a measure $Q$.
Finally, in Section \ref{sec:qst} we will prove the needed quenched asymptotics of slowdown probabilities for hitting times to complete the proof of Theorem \ref{thm0}.

\section{The Moment Generating Function of Hitting Times with an added Reflection Point}\label{sec:mgf}
In this section, we show an upper bound of the quenched moment generating function of hitting time with a reflection point.  We say a site $x$ is a reflection point if $\omega_x=1$.  Under our assumptions,
if $\a(\w_0 = 1) = 0$ (that is, there are no reflection points in the environment) then for $\a$-a.e.\ environment $\w$ the moment generating function $E_\omega[e^{\lambda\tau_1}]=\infty$ for all $\lambda>0$ \cite{Zei03}.  However, if we place a reflection point to the left of the starting point of the random walk then the moment generating function is finite for small enough $\l>0$ and we will give an upper bound for this modified moment generating function.
For any environment $\w$ and any $m\in \Z$, let $\w(m)$ be the environment $\w$ modified by adding a reflection point at $m$. That is,
\[
 \w(m)_x = \begin{cases}
            \w_x & x \neq m\\
            1 & x = m.
           \end{cases}
\]
The main result in this section is the following lemma which gives an upper bound on quenched moment generating functions of hitting times with a reflection point added to the left of the starting point.

\begin{lem}\label{lem-1}
Let $m \leq n$.  If $\l$ is small enough such that
\begin{equation}\label{lem1assum}
e^{-\lambda}-\sinh(\lambda)\left(E_{\omega(m)}^{m}[T_{n+1}]-(n+1-m)\right)>0
\end{equation}
where $\sinh(\l)=\frac{e^\l-e^{-\l}}{2}$, then for all $m\leq k\leq n$,
\begin{equation}
E_{\omega(m)}[e^{\lambda \tau_k}]\leq e^\lambda\frac{e^{-\lambda}-\sinh(\lambda)\left(E_{\omega(m)}^{m}[T_n]-(k-m)\right)}{e^{-\lambda}-\sinh(\lambda)\left(E_{\omega(m)}^{m}[T_{k+1}]-(k+1-m)\right)}.\label{lem1equation}
\end{equation}
\end{lem}
\begin{remark}\label{rem:mgfub}
 Since $E_{\w(m)}^{m}[T_{n+1}]-(n+1-m) = \sum_{k=m}^{n} ( E_{\w(n)}[\tau_k] - 1)$ is non-decreasing in $n$, if $\l>0$ is such that \eqref{lem1assum} holds then it follows that $e^{-\l}-\sinh(\l)\left(E_{\w(m)}^{m}[T_{k+1}]-(k+1-m)\right)>0$ for all $m\leq k\leq n$, and this is the condition that will be used in the proof below to obtain the upper bound \eqref{lem1equation}.
\end{remark}

\begin{proof}
Clearly, it is enough to prove the statement of the lemma when $m=0$. Therefore, for convenience of notation, let $g(k)=E_{\omega(0)}[e^{\lambda \tau_{k}}]$ for $k\geq 0$.  We need to show that 
\begin{equation}\label{cor02}
 g(k) \leq e^\lambda\frac{e^{-\lambda}-\sinh(\lambda)(E_{\omega(0)}[T_k]-k)}{e^{-\lambda}-\sinh(\lambda)(E_{\omega(0)}[T_{k+1}]-(k+1))}, \quad \text{for } 0\leq k \leq n,
\end{equation}
whenever $\l$ is small enough so that
\begin{equation}\label{lsmall}
 e^{-\l}-\sinh(\l)(E_{\omega(0)}[T_{n+1}]-n-1)>0.
\end{equation}
For $n=k=0$, $g(0)=e^\l$ because a reflection point to the right is placed at a site $0$.
Thus, \eqref{cor02} clearly holds when $n =0$ and so we need only to consider $n\geq 1.$  For any $1\leq k\leq n$, let us decompose $\tau_k$ into a series of crossing times from $k-1$ to $k$ before reaching $k+1$.  Let $N$ be a number of times a walk steps from $k$ to $k-1$ before stepping from $k$ to $k+1$.  Then, $N$ is a geometric random variable with a success probability of $\omega_k$ and
\[\tau_k=N+1+\sum_{i=1}^N\tau_{k-1}^{(i)}\text{ in distribution,}\]
where $\tau_{k-1}^{(i)}$ is an independent copy of $\tau_{k-1}$ for each $i$.  Therefore, we have that
\begin{align}
g(k)
&=\sum_{n=0}^{\infty}E_{\omega(0)}\left[e^{\l (N+1+\sum_{i=1}^N\tau_{k-1}^{(i)})}|N=n\right]P(N=n)\nonumber\\
&=\sum_{n=0}^{\infty}e^{\l(n+1)} g(k-1)^n (1-\omega_k)^n\omega_k\nonumber\\
&=\omega_k e^{\l}\sum_{n=0}^{\infty}\left((1-\omega_k)e^{\l} g(k-1) \right)^n.\nonumber
\end{align}
Here, we claim the following statement and postpone its proof until the end that \eqref{lsmall} is a sufficient condition for
\begin{equation}
(1-\omega_k)e^{\l}g(k-1)<1,
\quad 1\leq k\leq n.\label{lem1geom}
\end{equation}
Then with a sufficiently small $\l$, we obtain a representation of the moment generating function introduced in terms of continued fraction or
\begin{equation}\label{contfrac}
g(k)=\frac{\omega_k e^\lambda}{1-(1-\omega_k) e^\lambda g(k-1)},
\quad 1\leq k\leq n.
\end{equation}
Using  \eqref{contfrac}, we will give a proof of \eqref{cor02} by induction in $k$.  If $k=1$, then
$$g(1)=\frac{\omega_1 e^\lambda}{1-(1-\omega_1) e^\lambda g(0)}=\frac{\omega_1 e^\lambda}{1-(1-\omega_1) e^{2\lambda}}=\frac{1}{e^{-\lambda}+\rho_1 e^{-\lambda}-\rho_1 e^{\lambda}}=\frac{1}{e^{-\lambda}-\sinh(\lambda)(E_{\omega(0)}^0[T_2]-2)},$$
where the last equality is obtained by noting that \eqref{beta} implies $E_{\omega(0)}[T_2] = 2 + 2\rho_1$.
Suppose that the inequality in \eqref{cor02} holds for $g(k-1)$.  Then
\begin{align}
&g(k)=\frac{\omega_{k} e^\l}{1-(1-\omega_{k}) e^\l g(k-1)}=\frac{1}{e^{-\l}+\rho_{k}e^{-\lambda}-\rho_{k}g(k-1)} \nonumber \\
&\quad\leq\frac{1}{e^{-\lambda}+\rho_{k}e^{-\lambda}-\rho_{k}e^\lambda\left(\frac{e^{-\lambda}-\sinh(\lambda)(E_{\omega(0)}[T_{k-1}]-(k-1))}{e^{-\lambda}-\sinh(\lambda)(E_{\omega(0)}[T_{k}]-k)}\right)} \nonumber \\
&\quad=\frac{e^{-\lambda}-\sinh(\lambda)(E_{\omega(0)}[T_{k}]-k)}{(e^{-\lambda}+\rho_{k}e^{-\lambda})(e^{-\lambda}-\sinh(\lambda)(E_{\omega(0)}[T_{k}]-k))-\rho_{k}e^\lambda(e^{-\lambda}-\sinh(\lambda)(E_{\omega(0)}[T_{k-1}]-(k-1)))} \nonumber \\
&\quad=e^\l\frac{e^{-\l}-\sinh(\l)(E_{\omega(0)}[T_{k}]-k)}{(1+\rho_{k})(e^{-\l}-\sinh(\l)(E_{\omega(0)}[T_{k}]-k))-\rho_{k}e^{2\lambda}(e^{-\lambda}-\sinh(\l)(E_{\omega(0)}[T_{k-1}]-(k-1)))} \nonumber \\
&\quad \leq e^\l\frac{e^{-\l}-\sinh(\l)(E_{\omega(0)}[T_{k}]-k)}{(1+\rho_{k})(e^{-\l}-\sinh(\l)(E_{\omega(0)}[T_{k}]-k))-\rho_{k}(e^{\l}-\sinh(\l)(E_{\omega(0)}[T_{k-1}]-(k-1)))}. \label{mgfinduct}
\end{align}
The proof of \eqref{lem1equation} will then be complete if we can show the denominator in \eqref{mgfinduct} is equal to
the denominator in \eqref{cor02}.
To this end, note that \eqref{beta} implies that
\begin{equation}\label{crossform}
E_{\omega(0)}[T_k]=k+2\sum_{j=1}^{k-1}\sum_{i=1}^j\Pi_{i,j}.
\end{equation}
Therefore, the denominator in \eqref{mgfinduct} is equal to
\begin{align}
&(1+\rho_{k})\left(e^{-\lambda}-2 \sinh(\lambda)\sum_{j=1}^{k-1}\sum_{i=1}^j\Pi_{i,j} \right) -\rho_{k}\left(e^{\lambda}-2 \sinh(\lambda) \sum_{j=1}^{k-2}\sum_{i=1}^j\Pi_{i,j}\right)\nonumber \\
&\quad=e^{-\lambda}-\rho_{k}(e^{\lambda}-e^{-\lambda})-2\sinh(\lambda)\left(\sum_{j=1}^{k-1}\sum_{i=1}^j\Pi_{i,j}+\rho_{k}\sum_{j=1}^{k-1}\sum_{i=1}^j\Pi_{i,j}-\rho_{k}\sum_{j=1}^{k-2}\sum_{i=1}^j\Pi_{i,j}\right)\nonumber\\
&\quad=e^{-\lambda}-2 \sinh(\lambda)\left(\rho_{k}+ \sum_{j=1}^{k-1}\sum_{i=1}^j\Pi_{i,j}+ \sum_{i=1}^{k-1}\Pi_{i,k}\right)\nonumber\\
&\quad=e^{-\lambda}-2 \sinh(\lambda)\sum_{j=1}^{k}\sum_{i=1}^j\Pi_{i,j}\nonumber \\
&\quad=e^{-\lambda}-\sinh(\lambda)(E_{\omega(0)}[T_{k+1}]-(k+1)).\label{compeq}
\end{align}
Finally, it remains to prove that \eqref{lsmall} implies \eqref{lem1geom}.  The proof uses a mathematical induction in $k$ which is very similar to the proof of \eqref{cor02}.  If $k=1$, \eqref{lsmall} and Remark \ref{rem:mgfub} implies
$$e^{-\l}>\sinh(\l)(E_{\omega(0)}[T_2]-2) = (e^{\l} - e^{-\l}) \rho_1.$$
Since $\rho_1=(1-\omega_1)/\omega_1$ and $g(0)=e^\l$, this is equivalent to
$$1>e^{2\l}(1-\omega_1)
=e^{\l}g(0)(1-\omega_1).
$$
This verifies \eqref{lem1geom} for $k=1$.
Suppose now that \eqref{lem1geom} holds up to $k-1<n$.  Then, the above proof shows that the inequality
 \eqref{cor02} holds for $g(k-1)$. Therefore,
\begin{align*}
&1-e^{\l}g(k-1)(1-\omega_k)\\
&\quad \geq 1- e^{\l}(1-\omega_k)e^\lambda\frac{e^{-\lambda}-\sinh(\lambda)(E_{\omega(0)}[T_{k-1}]-(k-1))}{e^{-\lambda}-\sinh(\lambda)(E_{\omega(0)}[T_{k}]-k)}\\
&\quad = \frac{e^{-\lambda}-\sinh(\lambda)(E_{\omega(0)}[T_{k}]-k)-e^{2\l}(1-\omega_k)(e^{-\lambda}-\sinh(\lambda)(E_{\omega(0)}[T_{k-1}]-(k-1)))}{e^{-\lambda}-\sinh(\lambda)(E_{\omega(0)}[T_{k}]-k)}\\
&\quad\geq\frac{e^{-\lambda}-\sinh(\lambda)(E_{\omega(0)}[T_{k}]-k)-(1-\omega_k)(e^{\lambda}-\sinh(\lambda)(E_{\omega(0)}[T_{k-1}]-(k-1)))}{e^{-\lambda}-\sinh(\lambda)(E_{\omega(0)}[T_{k}]-k)}\\
&\quad\geq\omega_k\frac{(1+\rho_k)(e^{-\lambda}-\sinh(\lambda)(E_{\omega(0)}[T_{k}]-k))-\rho_k(e^{\lambda}-\sinh(\lambda)(E_{\omega(0)}[T_{k-1}]-(k-1)))}{e^{-\lambda}-\sinh(\lambda)(E_{\omega(0)}[T_{k}]-k)}\\
&\quad=\omega_k\frac{e^{-\lambda}-\sinh(\lambda)(E_{\omega(0)}[T_{k+1}]-(k+1))}{e^{-\lambda}-\sinh(\lambda)(E_{\omega(0)}[T_{k}]-k)},
\end{align*}
where the last equality comes from \eqref{compeq}.  Since $E_\a[\log \rho_0] < 0$ implies $\w_k>0$ and Remark \ref{rem:mgfub} implies
$$\frac{e^{-\lambda}-\sinh(\lambda)(E_{\omega(0)}[T_{k+1}]-(k+1))}{e^{-\lambda}-\sinh(\lambda)(E_{\omega(0)}[T_{k}]-k)}>0,$$
we get
that $1>e^{\l}g(k-1)(1-\omega_k)$.
\end{proof}


As a corollary of Lemma \ref{lem-1} we obtain the following upper bound for the quenched moment generating function of the time to cross an interval with a reflection point at some point to the left of $X_0$.
\begin{cor}\label{cor-1}
Suppose $m<k_0<k_1$ for any $m,k_0,k_1\in\mathbb{Z}.$  If $\l>0$ is sufficiently small enough such that
\begin{equation}
e^{-\l}-\sinh(\l)E_{\omega(m)}^{m}[T_{k_1}]>0,\label{cor-1assum}
\end{equation}
then,
\begin{equation}
E_{\omega(m)}^{k_0}[e^{\lambda T_{k_1}}]\leq\exp\left(\frac{\sinh(\lambda)E_{\omega(m)}^{k_0}[T_{k_1}]}{e^{-\l}-\sinh(\l)E_{\omega(m)}^m[T_{k_1}]}\right).\label{cor-1main}
\end{equation}
\end{cor}
\begin{proof}
First of all, if $\l>0$ is small enough so that \eqref{cor-1assum} holds then Remark \ref{rem:mgfub} implies that
\[
e^{-\l}-\sinh(\l)\left(E_{\omega(m)}^{m}[T_{i+1}]-(i+1-m)\right)>0,\text{ for all } k_0\leq i\leq k_1-1.
\]
By Lemma \ref{lem-1} and using the fact that the sequence $\{\tau_i\}_{k_0\leq i\leq k_1-1}$ is independent under $P_{\omega(m)}$, we have
\begin{align}
E_{\omega(m)}^{k_0}[e^{\lambda T_{k_1}}]&=E_{\omega(m)}[e^{\lambda \sum_{i=k_0}^{k_1-1}\tau_i}]
=\prod_{i=k_0}^{k_1-1}E_{\omega(m)}[e^{\l\tau_i}]\nonumber\\
&\leq\prod_{i=k_0}^{k_1-1}e^\lambda\frac{e^{-\lambda}-\sinh(\lambda)\left(E_{\omega(m)}^{m}[T_{i}]-(i-m)\right)}{e^{-\lambda}-\sinh(\lambda)\left(E_{\omega(m)}^{m}[T_{i+1}]-(i+1-m)\right)}\nonumber\\
&=e^{\l (k_1-k_0)}\frac{e^{-\lambda}-\sinh(\lambda)\left(E_{\omega(m)}^{m}[T_{k_0}]-(k_0-m)\right)}{e^{-\lambda}-\sinh(\lambda)\left(E_{\omega(m)}^{m}[T_{k_1}]-(k_1-m)\right)}\nonumber\\
&=e^{\l (k_1-k_0)}\left(1+\frac{\sinh(\lambda)\left(E_{\omega(m)}^{k_0}[T_{k_1}]-(k_1-k_0)\right)}{e^{-\lambda}-\sinh(\lambda)\left(E_{\omega(m)}^{m}[T_{k_1}]-(k_1-m)\right)}\right).\nonumber
\end{align}
Since $1+x\leq e^x$ for any $x\in\mathbb{R}$
we can conclude that
\begin{align*}
&e^{\l (k_1-k_0)}\left(1+\frac{\sinh(\lambda)\left(E_{\omega(m)}^{k_0}[T_{k_1}]-(k_1-k_0)\right)}{e^{-\lambda}-\sinh(\lambda)\left(E_{\omega(m)}^{m}[T_{k_1}]-(k_1-m)\right)}\right)\\
&\qquad\leq\exp\left(\l (k_1-k_0)+\frac{\sinh(\lambda)\left(E_{\omega(m)}^{k_0}[T_{k_1}]-(k_1-k_0)\right)}{e^{-\lambda}-\sinh(\lambda)\left(E_{\omega(m)}^{m}[T_{k_1}]-(k_1-m)\right)}\right)\\
&\qquad\leq\exp\left(\frac{\l (k_1-k_0)+\sinh(\lambda)\left(E_{\omega(m)}^{k_0}[T_{k_1}]-(k_1-k_0)\right)}{e^{-\lambda}-\sinh(\lambda)\left(E_{\omega(m)}^{m}[T_{k_1}]-(k_1-m)\right)}\right)\\
&\qquad\leq\exp\left(\frac{\sinh(\lambda)E_{\omega(m)}^{k_0}[T_{k_1}]}{e^{-\lambda}-\sinh(\lambda)E_{\omega(m)}^{m}[T_{k_1}]}\right),
\end{align*}
where in the second inequality we used that the denominator inside the exponent is at most $e^{-\l} \leq 1$, and in the last inequality we used that $\l < \sinh(\l)$ for $\l > 0$.
This completes the proof of the corollary.
\end{proof}


\section{Bounds for quenched expected crossing times}\label{sec:ald}
From the results of the previous section, we see that the quenched expected crossing times are key to obtaining bounds on the quenched moment generating functions of hitting times.
In particular, it will be necessary to obtain control on how small $\l>0$ must be for the bounds given by Corollary \ref{cor-1} to be valid.
In order to consider this problem in more general setting, let us define a sequence $a_n = n^{\eta_1}$ for some $\eta_1>0$, and study $E_\omega[T_{\nu_{a_n}}]$ under the measure $Q$.  First, we decompose $E_{\omega}[T_{\nu_{a_n}}]$ to the series of crossing time between consecutive ladder locations such that
$$E_{\omega}\left[T_{\nu_{a_n}}\right] = \sum_{i=0}^{\nu_{a_n-1}}E_{\omega}^{\nu_i}\left[T_{\nu_{i+1}}\right].$$ For simplicity, we will introduce some notation.
\[
 \beta_i = E_\omega^{\nu_i}\left[ T_{\nu_{i+1}} \right], \qquad i \in \mathbb{Z}.
\]
Under the measure Q, recall that $\nu_0 = 0$ and that $\theta^{\nu_i}\omega$ has a same distribution for any $i\in\mathbb{Z}$.  As a result, $\{\beta_i\}_{i\in\mathbb{Z}}$ is stationary under $Q$.  Next, we determine i.i.d components in $i$ which mainly contribute to the size of each $\beta_i$.  It turns out that $\beta_i$ is roughly comparable to $M_i$.  Suppose $b_n = n^{\eta_2}$ for some $\eta_1>\eta_2>0$. The main goal of this section is to show that the size of $\beta_i$ with $M_i\leq b_n$ is small enough that the sums of such $\beta_i$'s is unlikely to play a large role in the size of $E_{\omega}\left[T_{\nu_{a_n}}\right]$.  Therefore, the large deviation events are primarily dependent on the $\beta_i$ for indices $i$ with $M_i>b_n$.  The following Proposition is the main result of this section.
\begin{prop}\label{prop1}
Let $a_n= n^{\eta_1}$ and $b_n= n^{\eta_2}$ for some $\eta_1>\eta_2>0$.  Let Assumption \ref{asm:s} and \ref{asm:s2} hold.  Then, for any $\epsilon>0$ there exist constants $C,\, C'$ such that
\begin{equation}\label{prop1:ine}
Q\left(\sum_{i=0}^{a_n-1}(\beta_i\mathbf{I}_{\{M_i\leq b_n\}}-E_Q[\beta_0])>a_n \epsilon\right) \leq C'a_n e^{-C(\log n)^2}.
\end{equation}
\end{prop}

The remainder of this section is devoted to the proof of Proposition \ref{prop1}.  First of all, let $c_n := \lfloor(\log n)^2\rfloor$ and define $\beta_i^{(c_n)}$ to be a quenched expected crossing time from $\nu_i$ to $\nu_{i+1}$ with a reflection point located at $\nu_{i-(c_n-1)}$.  That is,
\[\b_i^{(c_n)}:=E_{\omega(\nu_{i-(c_n-1)})}^{\nu_i}\left[T_{\nu_{i+1}}\right].\]
The strategy of proof for \eqref{prop1:ine} is first to show that the sums of differences of $\beta_i\mathbf{I}_{\{M_i\leq b_n\}}$ and $\beta_i^{(c_n)}\mathbf{I}_{\{M_i\leq b_n\}}$ are negligible in the limit, and then prove the inequality of \eqref{prop1:ine} with $\b_i$ replaced by $\b_i^{(c_n)}$.  More precisely, we have
\begin{align}
&Q\left(\sum_{i=0}^{a_n-1}(\b_i\mathbf{I}_{\{M_i\leq b_n\}}-E_Q[\beta_0])>a_n \epsilon\right)\nonumber\\
&\quad\leq Q\left(\sum_{i=0}^{a_n-1}(\b_i-\b_i^{(c_n)})\mathbf{I}_{\{M_i\leq b_n\}}>\frac{\epsilon}{2} a_n\right) +  Q\left(\sum_{i=0}^{a_n-1}(\beta_i^{(c_n)}\mathbf{I}_{\{M_i\leq b_n\}}-E_Q[\beta_0])> \frac{\epsilon}{2}a_n\right),\label{prop1:ine2}
\end{align}
and we will show that each term in \eqref{prop1:ine2} is bounded above by $C a_n e^{-C'(\log n)^2}$.  The following lemma does this for the first term of \eqref{prop1:ine2}.
\begin{lem}\label{lem:beta0}
For any $\epsilon>0$, there exist $C,\,C'>0$ such that
$$Q\left(\sum_{i=0}^{a_n-1}(\b_i-\b_i^{(c_n)})\mathbf{I}_{\{M_i\leq b_n\}}>\frac{\epsilon}{2} a_n\right)<C'a_n e^{-C(\log n)^2}.$$
\end{lem}
\begin{proof}
Using \eqref{beta}, we may write
\begin{align}
\beta_i&=\sum_{j=\nu_i}^{\nu_{i+1}-1}(1+2W_j)\nonumber\\
&=l_i+2\sum_{j=\nu_i}^{\nu_{i+1}-1} W_{\nu_i,j}+2W_{\nu_i-1}R_{\nu_i,\nu_{i+1}-1}.\label{dec1}
\end{align}
Similarly, applying \eqref{beta} with a reflection at $\omega_{\nu_{i-(c_n-1)}}$ (so that $\rho_{\nu_{i-(c_n-1)}} = 0$) gives
\begin{equation}\label{beta:trunc}
\b_i^{(c_n)} = l_i+2\sum_{j=\nu_i}^{\nu_{i+1}-1}W_{\nu_i,j}+2R_{\nu_i,\nu_{i+1}-1}W_{\nu_{i-(c_n-1)},\nu_i-1}.
\end{equation}
Then by \eqref{dec1} and \eqref{beta:trunc}, we get
\[\b_i-\b_i^{(c_n)} = 2(1+W_{\nu_{i-{(c_n-1)}}-1})\Pi_{\nu_{i-(c_n-1)},\nu_i-1}R_{\nu_i,\nu_{i+1}-1}\]
Hence,
\begin{align*}
&Q\left(\sum_{i=0}^{a_n-1}\left(\beta_i-\beta_i^{(c_n)}\right)\mathbf{I}_{\{M_i\leq b_n\}}>a_n \frac{\epsilon}{2}\right)\\
&\quad= Q\left(\sum_{i=0}^{a_n-1}(1+W_{\nu_{i-{(c_n-1)}}-1})\Pi_{\nu_{i-(c_n-1)},\nu_i-1}R_{\nu_i,\nu_{i+1}-1}\mathbf{I}_{\{M_i\leq b_n\}}>a_n \frac{\epsilon}{4}\right)
\end{align*}
Since $\Pi_{i_1,i_2}\leq M_i$ for any $i_1,\,i_2$ such that $\nu_i\leq i_1\leq i_2\leq \nu_{i+1}-1,$
$$R_{\nu_i,\nu_{i+1}-1}\mathbf{I}_{\{M_i\leq b_n\}}=\sum_{k=\nu_i}^{\nu_{i+1}-1}\Pi_{\nu_i,k}\mathbf{I}_{\{M_i\leq b_n\}}\leq l_i M_i\mathbf{I}_{\{M_i\leq b_n\}}\leq l_i b_n.$$
Also, from \eqref{Qltails} and Lemma 2.2 in \cite{PZ09} there exist $c,\, c'>0$ such that
\begin{equation}\label{lwbound}
Q(l_0>x)< c e^{-c'x},\text{ and }Q(1+W_{-1}>x)< c e^{-c'x}.
\end{equation}
Applying \eqref{lwbound} and Chebyshev Inequality,
\begin{align*}
&Q\left(\sum_{i=0}^{a_n-1}(1+W_{\nu_{i-{(c_n-1)}}-1})\Pi_{\nu_{i-(c_n-1)},\nu_i-1}R_{\nu_i,\nu_{i+1}-1}\mathbf{I}_{\{M_i\leq b_n\}}>a_n \frac{\epsilon}{4}\right)\\
&\quad\leq Q\left(\exists i\in[0,a_n-1]:l_i>(\log n)^2\right) + Q\left(\exists i\in[-c_n+1,a_n-c_n]:1 + W_{\nu_i-1}>(\log n)^2\right)\\
&\qquad + Q\left(\sum_{i=0}^{a_n-1}\Pi_{\nu_{i-(c_n-1)},\nu_i-1}>\frac{\epsilon a_n}{4(\log n)^4b_n}\right)\\
&\quad\leq 2c a_n e^{-c'(\log n)2} + \frac{4(\log n)^4b_n}{\epsilon} E_Q[\Pi_{0,\nu_1-1}]^{c_n-1}\leq C a_n e^{-C(\log n)^2}\text{ for some }C,\,C'>0,
\end{align*}
where the second to last inequality comes from the fact that
$E_Q[\Pi_{\nu_{i-k},\nu_i-1}] = E_Q[\Pi_{0,\nu_1-1}]$ (since blocks between ladder points are i.i.d. under $Q$),
and the last inequality follows from
 $E_Q[\Pi_{0,\nu_1-1}]<1$, $b_n(\log n)^4\ll a_n$, and $c_n=\lfloor(\log n)^2\rfloor$.
\end{proof}
Regarding the second term of \eqref{prop1:ine2}, we will begin by decomposing $\b_i^{(c_n)}$ in a way that will help us to get control the dependence in the sequence.  Recall the decomposition of $\beta_i^{(c_n)}$ in \eqref{beta:trunc}.  Observe that the first two terms are i.i.d as sequences indexed by $i$, and the last term is stationary in $i$ but dependent under measure $Q$.  Since \eqref{beta} implies that
\[E_Q[\beta_0] = E_Q[l_0]+2E_Q\left[\sum_{j=0}^{\nu_1-1} W_{0,j}\right]+2E_Q[W_{-1}R_{0,\nu_1-1}],\]
we can bound the second term of \eqref{prop1:ine2} by three different probabilities such that
\begin{align}
&Q\left(\sum_{i=0}^{a_n-1}(\beta_i^{(c_n)}\mathbf{I}_{\{M_i\leq b_n\}}-E_Q[\beta_0])> \frac{\epsilon}{2}a_n\right)\nonumber\\
&\quad\leq Q\left(\sum_{i=0}^{a_n-1}(l_i-E_Q[l_0])>a_n\frac{\epsilon}{6}\right)\nonumber\\
&\qquad + Q\left(\sum_{i=0}^{a_n-1}\left(\sum_{j=\nu_i}^{\nu_{i+1}-1} W_{\nu_i,j}\mathbf{I}_{\{M_i<b_n\}}-E_Q\left[\sum_{j=\nu_0}^{\nu_1-1} W_{\nu_0,j}\right]\right)>\frac{\epsilon}{12}a_n \right)\nonumber\\
&\qquad + Q\left(\sum_{i=0}^{a_n-1}\left\{W_{\nu_{i-(c_n-1)},\nu_i-1}R_{\nu_i,\nu_{i+1}-1}\mathbf{I}_{\{M_i\leq b_n\}}-E_Q[W_{-1}R_{0,\nu_1-1}]\right\} >a_n\frac{\epsilon}{12}\right).\nonumber
\end{align}
The proof of Proposition \ref{prop1} then follows easily from the following three lemmas.
\begin{lem}\label{lem:beta1}
For any $\epsilon>0$, there exists $c(\epsilon)>0$ such that
$$Q\left(\sum_{i=0}^{a_n-1}(l_i-E_Q[l_0])>a_n \epsilon\right)=O\left(e^{-c(\epsilon)a_n}\right).$$
\end{lem}

\begin{lem}\label{lem:beta2}
For any $\epsilon>0$, there exists $C,\,C'>0$ such that
\begin{equation}
Q\left(\sum_{i=0}^{a_n-1}\sum_{j=\nu_i}^{\nu_{i+1}-1} W_{\nu_i,j}\mathbf{I}_{\{M_i<b_n\}}-E_Q\left[\sum_{j=\nu_0}^{\nu_1-1} W_{\nu_0,j}\right]>a_n \epsilon\right)\leq C a_n e^{-C'(\log n)^2}.\label{lem:beta21}
\end{equation}
\end{lem}

\begin{lem}\label{lem:beta3}
For any $\epsilon>0$, there exists constants $C,\, C'>0$ such that
\begin{equation}
Q\left(\sum_{i=0}^{a_n-1}\left\{W_{\nu_{i-(c_n-1)},\nu_i-1}R_{\nu_i,\nu_{i+1}-1}\mathbf{I}_{\{M_i\leq b_n\}}-E_Q[W_{-1}R_{0,\nu_1-1}]\right\} >a_n \epsilon \right)\leq C a_n e^{-C'(\log n)^2}\label{beta:30}
\end{equation}
\end{lem}
The proof of Lemma \ref{lem:beta1} is a standard result in large deviation theory since the $l_i$ are i.i.d.\ with exponential tails. We will therefore only give the proofs of Lemmas \ref{lem:beta2} and \ref{lem:beta3}.
Although the summands in \eqref{lem:beta21} are i.i.d, we cannot use
the standard large deviation techniques involving exponential moments to obtain a bound like in Lemma \ref{lem:beta1}
because the exponential moment is infinite as $\sum_{j=\nu_i}^{\nu_{i+1}-1} W_{\nu_i,j}>M_i$ and $Q(M_i>x)\sim C''/x^s$ for $1< s$.  Instead, we adapt a technique of Nagaev and Fuk used on estimating for large deviation probability of sums of heavy tailed independent random variables \cite{FN71}.
Let $X$ be a random variable on arbitrary probability space $\Omega$ and let $A$ be a measurable subset of $\Omega$.  If $X\leq y$, we claim that for any $h>0$,
\begin{equation}
E\left[e^{hX}-1-hX\right]\leq \frac{e^{hy}-1-hy}{y^2}E\left[X^2\right].\label{lem:intro2}
\end{equation}
It is easy to verify \eqref{lem:intro2} by the fact that $(e^{hx}-1-hx)/x^2$ is non-decreasing in $x$.
Secondly, we state the following lemma which follows easily from the tail asymptotics \eqref{M} for $M_i$ under the measure $Q$.
\begin{lem}\label{lem:EM2} Let Assumptions \ref{asm:s} and \ref{asm:s2} hold.
 \begin{enumerate}
  \item If $s<2$, then $E_Q[M_0^2 \mathbf{I}_{M_0 \leq x} ] \sim \frac{C_0 s}{2-s} x^{2-s}$ as $x\ra\infty$.
  \item If $s=2$, then  $E_Q[M_0^2 \mathbf{I}_{M_0 \leq x} ] \sim 2 C_0 \log x$ as $x\ra\infty$.
  \item If $s>2$, then $E_Q[M_0^2] < \infty$.
 \end{enumerate}
\end{lem}
Now we are ready to give the proof of Lemma \ref{lem:beta2}.
\begin{proof}[Proof of Lemma \ref{lem:beta2}]
For the simplicity, let us first introduce a notation.
$$W := E_Q\left[\sum_{j=\nu_0}^{\nu_1-1} W_{\nu_0,j} \right].$$
Also, define a positive function $\zeta(i,n)$ such that
$$\zeta(i,n):=\sum_{j=\nu_i}^{\nu_{i+1}-1} W_{\nu_i,j}\mathbf{I}_{\{M_i<b_n\}\cap \{l_i<(\log n)^2\}}.$$
Since $\Pi_{i_1,i_2}\leq M_i\text{ for any } i_1,i_2\text{ such that } \nu_i\leq i_1\leq i_2\leq \nu_{i+1}-1,$
\[
\sum_{j=\nu_i}^{\nu_{i+1}-1} W_{\nu_i,j}=\sum_{j=\nu_i}^{\nu_{i+1}-1}\sum_{k=\nu_i}^j\Pi_{k,j}\leq \sum_{j=\nu_i}^{\nu_{i+1}-1}\sum_{k=\nu_i}^j M_i\leq \sum_{j=\nu_i}^{\nu_{i+1}-1} l_i M_i\leq l_i^2 M_i.
\]
As a result, we have a following bound of $\zeta(i,n)$.
\begin{equation}
\zeta(i,n)\leq (\log n)^4 M_i\mathbf{I}_{\{M_i<b_n\}}\leq (\log n)^4b_n.\label{func:zeta}
\end{equation}
Replacing the notations in the problem by $\zeta(i,n)$ and $W$ the notation above, the problem is simplified to
\begin{align}
&Q\left(\sum_{i=0}^{a_n-1}\left(\sum_{j=\nu_i}^{\nu_{i+1}-1} W_{\nu_i,j}\mathbf{I}_{\{M_i<b_n\}}-W\right)>a_n \epsilon\right)\nonumber\\
&\qquad\leq Q\left(\exists i\in[0,a_n-1]:l_i>(\log n)^2\right) + Q\left(\sum_{i=0}^{a_n-1}\zeta(i,n) > a_n(\epsilon+W)\right). \label{lem:beta20}
\end{align}
By \eqref{lwbound} and the stationarity of $l_i$ under Q, the first term is bounded by $c a_n e^{-c'(\log n)^2}$ for some $c,\,c'>0.$  So, it remains to prove a similar upper bound for the second term of \eqref{lem:beta20}.  Recall that $\zeta(i,n)$ is i.i.d. sequences in $i$ under the measure Q.  Then by Chebyshev Inequality, for any $\l\geq 0$
\begin{equation}
Q\left(\sum_{i=0}^{a_n-1}\zeta(i,n)>a_n(\epsilon+W)\right)\leq E_Q[e^{\lambda\sum_{i=0}^{\lceil a_n\rceil-1}\zeta(i,n)}]e^{-\lambda a_n(W + \epsilon)}=e^{-\lambda a_n (\epsilon+W)} E_Q[e^{\lambda\zeta(0,n)}]^{\lceil a_n\rceil}.\label{lem:beta22}
\end{equation}
Note that $E_Q[\zeta(0,n)]\leq W$, and $\zeta(0,n)\leq (\log n)^4b_n$.  Then using \eqref{lem:intro2} and \eqref{func:zeta},
\begin{align}
E_Q\left[e^{\lambda\zeta(0,n)}\right]&= 1+\lambda E_Q\left[\zeta(0,n)\right]+E_Q\left[e^{\lambda\zeta(0,n)}-1-\lambda\zeta(0,n)\right]\nonumber\\
&\leq 1+\l W+\frac{e^{\lambda (\log n)^4b_n}-1-\lambda (\log n)^4b_n}{((\log
n)^4b_n)^2}E\left[\zeta(0,n)^2\right]\nonumber\\
&\leq 1+\l W+\frac{e^{\lambda (\log n)^4b_n}-1-\lambda (\log n)^4b_n}{b_n^2}E\left[M_0^2 \mathbf{I}_{\{M_0<b_n\}}\right]\nonumber
\end{align}
With a choice of $\lambda=\frac{1}{(\log n)^4b_n}$, we get
$$E_Q\left[e^{\lambda\zeta(0,n)}\right]\leq 1+\frac{W}{(\log n)^4b_n}+\frac{2}{b_n^2}E\left[M_0^2 \mathbf{I}_{\{M_0<b_n\}}\right]\leq \exp\left(\frac{W}{(\log n)^4b_n}+\frac{2}{b_n^2}E\left[M_0^2 \mathbf{I}_{\{M_0<b_n\}}\right]\right).$$
Applying Lemma \ref{lem:EM2}, we obtain that
there exists some constant $C>0$ and
\begin{equation}\label{cases:zeta}
E_Q\left[e^{\lambda\zeta(0,n)}\right]\leq
\begin{cases}
\exp\left(\frac{W}{(\log n)^4b_n}+\frac{C}{b_n^s}\right)&\text{ if }1<s<2\\
\exp\left(\frac{W}{(\log n)^4b_n}+\frac{C\log n}{b_n^2}\right)&\text{ if }s=2\\
\exp\left(\frac{W}{(\log n)^4b_n}+\frac{C}{b_n^2}\right)&\text{ if }2<s,
\end{cases}
\qquad \text{when } \l=\frac{1}{(\log n)^4b_n}.
\end{equation}
Combining \eqref{lem:beta22} and \eqref{cases:zeta} we get that there exists a constant $c>0$ such that
\[
Q\left(\sum_{i=0}^{a_n-1}\zeta(i,n)>a_n(\epsilon+W)\right)\leq
\begin{cases}
c\times\exp\left(a_n\left(\frac{C}{b_n^s}-\frac{\epsilon}{(\log n)^4b_n}\right)\right)&\text{ if }1<s<2\\
c\times\exp\left(a_n\left(\frac{C\log n}{b_n^2}-\frac{\epsilon}{(\log n)^4b_n}\right)\right)&\text{ if }s=2\\
c\times\exp\left(a_n\left(\frac{C}{b_n^2}-\frac{\epsilon}{(\log n)^4b_n}\right)\right)&\text{ if }2<s.
\end{cases}
\]
Note that all three cases are bounded above by $c e^{-c' \frac{a_n}{(\log n)^4 b_n}}$ for some $c,\, c'>0$ for $n$ large enough.  Hence, the second term of \eqref{lem:beta20} is bounded above by the right-hand side of \eqref{lem:beta21} for large $n$.
\end{proof}

In preparation for the proof of Lemma \ref{lem:beta3}, we introduce the following notation.
$$\tilde{W}_{i,n}:=W_{\nu_{i-(c_n-1),\nu_i-1}}\text{ and }\tilde{R}_i:=R_{\nu_i,\nu_{i+1}-1},$$
and define
\[\psi(i,n) := \tilde{R}_i\mathbf{I}_{\{M_i\leq b_n, \, l_i\leq(\log n)^2\}}\tilde{W}_{i,n}\mathbf{I}_{\{\tilde{W}_{i,n}<(\log n)^2\}}.\]
 Note that $\psi(i+c_n,n)$ is independent of $\psi(i,n)$ under the measure $Q$. Also, since $\tilde{R}_i=\sum_{k=\nu_i}^{\nu_{i+1}-1}\Pi_{\nu_i,k}\leq l_iM_i$,
\begin{equation}\label{psiub}
\psi(i,n)\leq (\log n)^4M_i\mathbf{I}_{\{M_i<b_n\}}\leq (\log n)^4 b_n.
\end{equation}
Finally, we give the proof of Lemma \ref{lem:beta3}.
\begin{proof}[Proof of Lemma \ref{lem:beta3}]
For the simplification to notation, denote $W' := E_Q[W_{-1}]$ and $R:=E[R_{0,\nu_1-1}]$. Note that $R_{0,\nu_1-1}$ and $W_{-1}$ are independent because $R_{0,\nu_1-1} \in\{\omega_x: 0\leq x\leq \nu_1-1 \}$ while $W_{-1}\in\{\omega_x:x\leq -1\}$, so we get
$$E_Q[W_{-1}R_{0,\nu_1-1}] = E_Q[W_{-1}] E_Q[R_{0,\nu_1-1}] = W'R.$$
With new notations described above, the problem is simplified to
\begin{align}
&Q\left(\sum_{i=0}^{a_n-1}\left\{\tilde{W}_{i,n}\tilde{R_i}-W'R\right\} >a_n\epsilon\right)\leq Q\left(\exists i\in[0,a_n-1]:l_i>(\log n)^2\right)\nonumber\\
&\qquad + Q\left(\exists i\in[0,a_n-1]:W_{\nu_i-1}>(\log n)^2\right) + Q\left(\sum_{i=0}^{a_n-1}(\psi(i,n)-W'R)>a_n\epsilon\right).\label{beta:31}
\end{align}
By \eqref{lwbound}, the first and second terms of \eqref{beta:31} are bounded by $2c a_n e^{-c'(\log n)^2}$ for some constant $c,\,c'>0$.  Hence, it is enough to show that there exist some constants $C, C'>0,$ such that
\begin{equation}\label{func:psi}
Q\left(\sum_{i=0}^{a_n-1}\psi(i,n)>a_n(\epsilon+W'R)\right)\leq C a_n e^{-C'(\log n)^2}
\end{equation}
A proof of \eqref{func:psi} begins with grouping $\{\psi(i,n)\}_{\{0\leq i\leq a_n-1\}}$ into $c_n = \fl{(\log n)^2}$ smaller sums as follows.
In particular, since
\[
\sum_{i=0}^{a_n-1}\psi(i,n)
\leq \sum_{j=0}^{c_n-1}\left(\sum_{i=0}^{\lfloor a_n/c_n\rfloor}\psi(j+i\, c_n,n)\right),
\]
then the third term of \eqref{beta:31} is bounded above by
\begin{align}
Q\left(\sum_{j=0}^{c_n-1}\left(\sum_{i=0}^{\lfloor a_n/c_n\rfloor}\psi(j+i\, c_n,n)\right)>a_n(\epsilon+W'R)\right)&\leq \sum_{j=0}^{c_n-1} Q\left(\sum_{i=0}^{\lfloor a_n/c_n\rfloor}\psi(j+i\, c_n,n) >\frac{a_n}{c_n}(\epsilon+W'R)\right)\nonumber\\
&= c_n
Q\left(\sum_{i=0}^{\lfloor a_n/c_n\rfloor}\psi(i\, c_n,n) >\frac{a_n}{c_n}(\epsilon+W'R)\right),\label{psi:indep}
\end{align}
where the last equality follows by the stationarity of $\psi(i,n)$ under $Q$.   Notice terms in the sum inside the probability in \eqref{psi:indep} are i.i.d. under $Q$.  Hence, applying Chebyshev Inequality to \eqref{psi:indep}, for any $\l>0$
\begin{align}
Q\left(\sum_{i=0}^{\lfloor a_n/c_n\rfloor}\psi(i\, c_n,n) >\frac{a_n}{c_n}(\epsilon+W'R)\right)&\leq E_Q\left[e^{\sum_{i=0}^{\lfloor a_n/c_n\rfloor}\lambda \psi(i\, c_n,n)}\right]e^{-\frac{\lambda a_n}{c_n}\left(\epsilon+W'R\right)}\nonumber\\
&= E_Q\left[e^{\lambda \psi(0,n)}\right]^{\lfloor a_n/c_n\rfloor+1}e^{-\frac{\lambda a_n}{c_n}\left(\epsilon+W'R\right)}. \label{psisums}
\end{align}
Note that $E[\psi(0,n)]\leq W'R$ and recall that $\psi(0,n)\leq (\log n)^4b_n$.
Therefore, using \eqref{lem:intro2} 
\begin{align}
E_Q\left[e^{\lambda\psi(0,n)}\right]&= 1+\lambda E_Q\left[\psi(0,n)\right]+E_Q\left[e^{\lambda\psi(0,n)}-1-\lambda\psi(0,n)\right]\nonumber\\
&\leq 1+\l W'R+\frac{e^{\lambda (\log n)^4b_n}-1-\lambda (\log n)^4b_n}{((\log
n)^4b_n)^2}E\left[\psi(0,n)^2\right]\nonumber\\
&\leq 1+\l W'R+\frac{e^{\lambda (\log n)^4b_n}-1-\lambda (\log n)^4b_n}{b_n^2}E\left[M_{0}^2 \mathbf{I}_{\{M_{0}<b_n\}}\right],\nonumber
\end{align}
where in the last line we used the first inequality in \eqref{psiub}.
With a choice of $\lambda=\frac{1}{(\log n)^4b_n},$ we get
\begin{align*}
E_Q\left[e^{\lambda\psi(0,n)}\right]&\leq 1+\frac{W'R}{(\log n)^4b_n}+\frac{2}{b_n^2}E\left[M_{0}^2 \mathbf{I}_{\{M_{0}<b_n\}}\right]\\
&\leq \exp\left(\frac{W'R}{(\log n)^4b_n}+\frac{2}{b_n^2}E\left[M_{0}^2 \mathbf{I}_{\{M_{0}<b_n\}}\right]\right),
\end{align*}
and thus, applying Lemma \ref{lem:EM2}, there exists a constant $C>0$ such that
\begin{equation}\label{cases:psi}
E_Q\left[e^{\lambda\psi(0,n)}\right]\leq
\begin{cases}
\exp\left(\frac{W'R}{(\log n)^4b_n}+\frac{C}{b_n^s}\right)&\text{ if }1<s<2\\
\exp\left(\frac{W'R}{(\log n)^4b_n}+\frac{C\log n}{b_n^2}\right)&\text{ if }s=2\\
\exp\left(\frac{W'R}{(\log n)^4b_n}+\frac{C}{b_n^2}\right)&\text{ if }2<s,
\end{cases}
\qquad \text{when } \l=\frac{1}{(\log n)^4b_n}.
\end{equation}
Combining \eqref{psi:indep}, \eqref{psisums} and \eqref{cases:psi}, there exist constants $c,C>0$ such that, for large $n$,
\[Q\left(\sum_{i=0}^{a_n-1}(\psi(i,n)-W'R)>a_n\epsilon\right)\leq
\begin{cases}
c\times c_n\times\exp\left(\frac{a_n}{c_n}\left(\frac{C}{b_n^s}-\frac{\epsilon}{( \log n)^4b_n}\right)\right)&\text{ if }1<s<2\\
c\times c_n\times\exp\left(\frac{a_n}{c_n}\left(\frac{C\log n}{b_n^2}-\frac{\epsilon}{(\log n)^4b_n}\right)\right)&\text{ if }s=2\\
c\times c_n\times\exp\left(\frac{a_n}{c_n}\left(\frac{C}{b_n^2}-\frac{\epsilon}{( \log n)^4b_n}\right)\right)&\text{ if }2<s.
\end{cases}
\]
Note that all three cases are bounded above by $c\times c_n e^{-c' \frac{a_n}{c_n(\log n)^4 b_n}}<Ca_ne^{-C'(\log n)^2}$ for some $c,\, c',\,C,\, C' >0$ with $n$ large enough, which completes the proof of \eqref{func:psi}.
\end{proof}

\section{The Quenched Subexponential Tail of Hitting Time Large Deviations}\label{sec:qst}
The main goal of this section is to prove the following.
\begin{prop}\label{prop2}
Under the same assumptions as Theorem \ref{thm0}, for any $u\in(\frac{1}{ v_\alpha},\infty),$
\begin{equation}
\liminf_{n\to\infty}\frac{1}{n^{1-1/s}}\log P_\omega(T_{\nu_n}>u\nu_n)=-\infty, \qquad \alpha\text{-a.s.}\label{prop21}
\end{equation}
\end{prop}
\noindent
Before giving the proof of Proposition \ref{prop2}, we will first show how it can be used to complete the proof of Theorem \ref{thm0}.
\begin{proof}[\textit{Proof of Theorem }\ref{thm0}]
Let $v<v'<v_\alpha$, then
\begin{equation}
P_\omega(X_n<nv)\leq P_\omega(T_{nv'}>n)+P_\omega^{nv'}(T_{nv}<\infty). \label{tm0}
\end{equation}
First, we will show that Proposition \ref{prop2} implies that
\begin{equation}\label{Tnvsubseq}
 \liminf_{n\ra\infty} \frac{1}{n^{1-1/s}}\log P_\omega(T_{nv'}>n)=-\infty, \quad \a\text{-a.s.}
\end{equation}
To this end, let $\mu$ and $\mu'$ be such that $v' < \mu < \mu' < v_\alpha$ and let $c_1 =  \frac{\mu}{E_Q[\nu_1]}$.
Since $\lim_{n\to\infty}\frac{\nu_n}{n}=E_Q[\nu_1]$, $\a$-a.s., it follows that
\[
 \lim_{n\to\infty}\frac{v_{\fl{c_1n}}}{n} = c_1 E_Q[\nu_1]  =\mu.
\]
That is, $ \frac{\nu_{\fl{c_1 n}}}{n} \in (v', \mu')$ for all $n$ sufficiently large (depending on $\w$).
Thus, for $\a$-a.e.\ environment and all $n$ large enough we have that
\[
 P_{\omega}(T_{nv'}>n)\leq P_{\omega}(T_{\nu_{\fl{c_1n}}}>n)=P_{\omega}\left(T_{\nu_{\fl{c_1n}}}>\frac{n}{\nu_{\fl{c_1n}}}\nu_{\fl{c_1n}}\right)\leq P_{\omega}\left(T_{\nu_{\fl{c_1n}}}>\frac{1}{\mu'}\nu_{\fl{c_1n}}\right),
\]
and since $1/\mu' > 1/v_\alpha$ it follows from Proposition \ref{prop2} that \eqref{Tnvsubseq} holds.
Regarding the second term on the right of \eqref{tm0}, it was shown in  \cite[Lemma 3.3]{GS02} that there is some constant $C>0$ such that $\P[T_m<\infty]\leq\exp(\textit{C}m)$ for any $m<0$. Therefore, we have a following upper bound with a choice of small $\epsilon>0$ such that
\begin{align*}
\P(P_\omega^{nv'}(T_{nv}<\infty)\geq e^{-\epsilon n})&\leq e^{\epsilon n}\P^{nv'}(T_{nv}<\infty)\\
&=e^{\epsilon n}\P(T_{n(v-v')}<\infty)
\leq e^{\epsilon n}e^{\textit{C}n(v-v')}.
\end{align*}
Since $v<v'$, if $\epsilon>0$ is chosen sufficiently small then the upper bound given above is exponentially decreasing in $n$ and so the Borel-Cantelli Lemma implies that $P_\w^{nv'}(T_{nv} < \infty )$ is almost surely eventually less than $e^{-C' n}$ for some constant $C'>0$ for all $n$ large.
In particular, this implies that
$$\lim_{n\to\infty}\frac{1}{n^{1-1/s}}\log P_\omega^{nv'}(T_{nv}<\infty)=-\infty,\qquad \alpha\text{-a.s.,}$$
which concludes our proof.
\end{proof}
To prove Proposition \ref{prop2} let us first define a new measure
$\tilde{\alpha}$ on environments by  $\tilde{\alpha}(\omega \in \cdot ) = \alpha( \theta^{\nu_0}\omega \in \cdot)$.
That is, $\alpha$ is the distribution of the environment shifted so that the ladder point $\nu_0 \leq 0$ is at the origin. Compare this with the distribution $Q$ which is obtained instead by \emph{conditioning} $\nu_0$ to be at the origin.
We show next that $\tilde{\alpha}$ is in fact absolutely continuous with respect to $Q$.
\begin{lem}\label{PQ}
$\tilde{\alpha}$ is absolutely continuous with respect to $Q$.
\end{lem}
\begin{proof}
First of all, note that
\[
 \{\nu_0 = -k\} = \left\{ \Pi_{j,-k-1} < 1 \text{ for } j < -k , \, \Pi_{-k,j} \geq 1 \text{ for } -k\leq j\leq -1 \right\}.
\]
Therefore, for any event $A\in\sigma(\{\omega_z\},z\in \mathbb{Z})$,
\begin{align*}
\tilde{\alpha}(\omega \in A)
&=\sum_{k=0}^\infty \alpha(\nu_0=-k)\alpha(\theta^{-k}\omega\in A \, | \, \nu_0=-k)\\
&=\sum_{k=0}^\infty \alpha(\nu_0=-k)\alpha(\theta^{-k}\omega\in A\, | \,\Pi_{j,-k-1} < 1 \text{ for } j < -k , \, \Pi_{-k,j} \geq 1 \text{ for } -k\leq j\leq -1 )\\
&= \sum_{k=0}^\infty \alpha(\nu_0=-k)\alpha( \omega\in A\, | \,\Pi_{j,-1} < 1 \text{ for } j < 0 , \, \Pi_{0,j} \geq 1 \text{ for } 0\leq j\leq k-1 )\\
&= \sum_{k=0}^\infty \alpha(\nu_0=-k) \frac{\alpha( \omega\in A, \, \Pi_{0,j} \geq 1 \text{ for } 0\leq j\leq k-1 \, | \, \Pi_{j,-1} < 1 \text{ for } j < 0) }{ \alpha( \Pi_{0,j} \geq 1 \text{ for } 0\leq j\leq k-1 \, | \, \Pi_{j,-1} < 1 \text{ for } j < 0  ) } \\
&= \sum_{k=0}^\infty \alpha(\nu_0=-k) \frac{Q( \omega \in A, \, \nu_1 > k )}{ Q(\nu_1 > k)}.
\end{align*}
Therefore, if $Q(\omega \in A) = 0$ then $\tilde{\alpha}(\omega \in A) = 0$ also. That is, $\tilde{\alpha}$ is absolutely continuous with respect to $Q$.
\end{proof}
\begin{remark}
 In fact, the above proof shows that $\frac{d\tilde{\alpha}}{dQ}(\omega) = \sum_{k=0}^{\nu_1(\omega)-1} r_k$, where $r_k = \frac{\alpha(\nu_0 = -k)}{Q(\nu_1 > k)}$.
\end{remark}

We now show how the measure $\tilde\alpha$ is helpful for proving Proposition \ref{prop2}. Since $\nu_0\leq 0$ for any environment $\omega$, we have
$$P_\omega(T_{\nu_n}>u\nu_n)\leq P^{\nu_0}_\omega(T_{\nu_n}>u\nu_n),
$$
and thus to prove Proposition \eqref{prop2} it will be enough to show that the conclusion holds with $\tilde{\alpha}$ in place of $\alpha$. That is, we need to show that
$$\liminf_{n\to\infty}\frac{1}{n^{1-1/s}}\log P_\omega(T_{\nu_n}>u\nu_n)=-\infty, \qquad \tilde{\alpha}\text{-a.s.}$$
However, since Lemma \ref{PQ} shows that $\tilde{\alpha}$ is absolutely continuous with respect to $Q$, the above limit will follow if we can show the same almost sure limit under the measure $Q$.
That is, we have reduced the proof of Proposition \ref{prop2} to the following.
\begin{prop}\label{propQ}
Under the same assumptions as Theorem \ref{thm0}, for any $u\in(\frac{1}{ v_\alpha},\infty),$
\begin{equation}
\liminf_{n\to\infty}\frac{1}{n^{1-1/s}}\log P_\omega(T_{\nu_n}>u{\nu_n})=-\infty, \qquad Q\text{-a.s.}\label{Qsubexp}
\end{equation}
\end{prop}

The remainder of the paper is devoted to the proof of Proposition \ref{propQ}.
We will follow the approach of \cite{Zei98} by dividing the environment into large blocks and then analyzing the crossing times of these large blocks. The main improvement we make is that we obtain better estimates on the quenched moment generating functions of these crossing times using the results from Section \ref{sec:mgf}.
To decompose the environment into blocks,
fix an integer $m>s$, let $n_k=m^{m^k}$ for $k\geq 0$ and let $a_k = n_k^{1/s}/D$ for some fixed $D>1$ which we will later allow to be arbitrarily large.
The blocks of the environment will be the intervals between ladder locations $\nu_{j a_k}$ and $\nu_{(j+1)a_k}$ for $j \in \Z$.
To simplify notation, let us denote the ladder locations at the edges of the blocks by
$$\nu(j,k):=\nu_{j a_k}, \qquad j\in \mathbb{Z}, \, k\geq 1.$$

The path of the random walk $X_n$ on $\mathbb{Z}$ naturally defines a birth-death chain by observing how the random walk moves from one $\nu(j,k)$ to either $\nu(j-1,k)$ or $\nu(j+1,k)$.
To be precise, let $\{t_i\}_{i\geq 0}$ be the sequence of times when the random walk reaches a ladder point $\nu(j,k)$ different from the last such ladder point visited. That is, $t_0 = 0$ and
\[
 t_i = \inf\left\{ n > t_{i-1}: X_n \in \{\nu(j,k)\}_{j\in \mathbb{Z}} \text{ and } X_n \neq X_{t_{i-1}} \right\}, \quad i \geq 1.
\]
We then obtain a birth-death process $\{Z_i\}_{i\geq 0}$ on $\mathbb{Z}$ by letting $X_{t_i} = \nu(Z_i,k)$.
If we let $\Theta_i = t_i - t_{i-1}$, then it follows that
\[
 T_{\nu_{n_k}} \leq \sum_{i=1}^{N_k} \Theta_i,
\]
where $N_k = \inf\{ i\geq 1: \, Z_i \geq n_k/a_k \}$ is the time needed for the induced birth-death process to move at least $n_k/a_k$ to the right.
If we also define $\tilde{N}_k = \inf\{ i\geq 1: \, |Z_i| \geq n_k/a_k \}$ to be the time for the birth-death process to exit $(-n_k/a_k,n_k/a_k)$ then it follows for any fixed $L$ that
\begin{equation}\label{main:inequal}
P_\omega( T_{\nu_{n_k}} > u \nu_{n_k} ) \leq
 P_\omega(N_k \neq \tilde{N}_k) +  P_\omega( \tilde{N}_k > L, \,  N_k = \tilde{N}_k ) + P_\omega\left( \sum_{i=1}^{\tilde{N}_k} \Theta_i > u \nu_{n_k}, \, \tilde{N}_k \leq L \right).
\end{equation}
We will show below that the environment is such that for $k$ large enough the induced birth-death process has a very strong drift to the right so that by choosing $L$ large enough we can make the first two probabilities on the right above very small.
The last probability on the right is the key term, and we will obtain control on this by obtaining certain uniform upper bounds on the time it takes a random walk started at $\nu(j,k)$ to reach either $\nu(j-1,k)$ or $\nu(j+1,k)$.

The following result shows that the first term in \eqref{main:inequal} has an exponential tail.

\begin{lem}\label{main:lemma1}
There exist $\delta>0$ such that for $Q$-a.e.\ environment $\omega$ there is an integer $K(\omega)<\infty$ such that
$$P_\omega(N_k \neq \tilde{N}_k)\leq e^{-\delta n_k}, \quad \forall k \geq K(\omega).$$
\end{lem}
\begin{proof}
The event $$\{ N_k \neq \tilde{N}_k \} \subset \{ T_{\nu_{-n_k}} < \infty \} \subset \{T_{-n_k} < \infty\}.$$ Therefore,
\begin{equation}
 Q\left( P_\w( N_k \neq \tilde{N}_k ) > e^{-\delta n_k} \right) \leq Q\left( P_\w( T_{-n_k} < \infty ) > e^{-\delta n_k} \right)
\leq e^{\d n_k} E_Q\left[ P_\w( T_{-n_k} < \infty ) \right].\label{lem1bc}
\end{equation}
Since $\alpha(\nu_0=0)>0$ and $Q(\cdot)=\alpha(\cdot|\nu_0=0)$, we have
\begin{equation}\label{Qbtprob}
E_Q\left[ P_\w( T_{-n_k} < \infty ) \right]=\frac{E_\a\left[ P_\w(T_{-n_k} < \infty) \mathbf{1}_{\nu_0=0} \right]}{\a(\nu_0= 0)} \leq\frac{\P(T_{-n_k} < \infty)}{\alpha(\nu_0=0)} \leq \frac{e^{-C n_k}}{\alpha(\nu_0=0)},
\end{equation}
where the last inequality holds by Lemma 3.3 in \cite{GS02}.
Finally, if $\d>0$ is chosen sufficiently small then \eqref{lem1bc} is summable in $k$ and so the Borel-Cantelli Lemma completes the proof.
\end{proof}
In order to determine the decay rate of the second and third term in \eqref{main:inequal}, we first define a set
\[
 J_{n_k}
= [ -n_k/a_k, \, n_k/a_k] \cap \Z
= \left\{-\left\lfloor \frac{n_k}{a_k} \right\rfloor, -\left\lfloor \frac{n_k}{a_k} \right\rfloor+1,\ldots, \left\lfloor \frac{n_k}{a_k} \right\rfloor \right\}.
\]
Clearly, if $N_k=\tilde{N}_k$ then the birth-death process $Z_i \in J_{n_k}$ when $t_i < T_{\nu_{n_k}}$.
So, we only need to observe paths of the birth-death process $\{Z_i\}_{i\geq0}$ restricted to $J_{n_k}$ and analyze its associated probability.  The following lemma gives a uniform upper bound (for all $k$ large enough) on the probability that the birth-death process steps to the left before time $\tilde{N}_k$.

\begin{lem}\label{lm31}
There exist $\delta'>0$ such that
\begin{equation}
Q\left(\max_{j \in J_{n_k}} P_\omega^{\nu(j,k)}( T_{\nu(j-1,k)} < T_{\nu(j+1,k)} ) > e^{-\delta' a_k}\quad i.o.\right)=0.\label{bc1}
\end{equation}
\end{lem}
\begin{proof}
First, note that
\begin{align*}
 &Q\left(\max_{j \in J_{n_k}} P_\omega^{\nu(j,k)}( T_{\nu(j-1,k)} < T_{\nu(j+1,k)} ) > e^{-\delta' a_k}\right)\\
&\qquad \leq \sum_{j \in J_{n_k}} Q\left( P_\omega^{\nu(j,k)}( T_{\nu(j-1,k)} < T_{\nu(j+1,k)} ) > e^{-\delta' a_k}\right)\\
&\qquad \leq 3 \frac{n_k}{a_k} Q\left( P_\omega( T_{\nu(-1,k)} < T_{\nu(1,k)} ) > e^{-\delta' a_k} \right) \\
&\qquad \leq 3 \frac{n_k}{a_k} Q\left( P_\omega( T_{-a_k} < \infty ) > e^{-\delta' a_k}\right) \\
&\qquad \leq 3 \frac{n_k}{a_k} E_Q\left[ P_\omega( T_{-a_k} < \infty ) \right] e^{\delta' a_k},
\end{align*}
where the second inequality holds because $|J_{n_k}|\leq 3n_k/a_k$ and $Q$ is stationary under shifts of the ladder points of the environment, and the third inequality holds by $\{T_{\nu(-1,k)}<T_{\nu(1,k)}\}\subseteq\{T_{\nu_{-a_k}} < \infty \}\subseteq\{T_{-a_k} < \infty \}$.
Finally, it follows from \eqref{Qbtprob} that the last line is bounded above by $C' \frac{n_k}{a_k} e^{-(C-\d')a_k}$. Since this is summable in $k$ for sufficiently small $\d'>0$, the Borel-Cantelli Lemma finishes the proof of \eqref{bc1}.
\end{proof}

Let $\{S_i\}_{i\geq 0}$ be a simple random walk with
$$P(S_{i+1}=S_i+1|S_i)=1-P(S_{i+1}=S_i-1|S_i)=1-e^{-\delta' a_k}.$$
Since this random walk steps to the right with very high probability, it is unlikely that the random walk takes too long to travel $\fl{n_k/a_k}$ steps to the right.
In particular, if we fix $\d>0$ and let $L_k = \frac{n_k}{a_k(1-\d)}$ then it was shown in \cite[Lemma 9]{Zei98} that
\[
 P\left(\inf \left\{i>0:S_i= \left\lceil \frac{n_k}{a_k}\right\rceil  \right\}>L_k\right)\leq e^{-\delta_1 n_k}.
\]
for some $\d_1>0$.
It follows from Lemma \ref{lm31} that the probability of jumping to left under $S_i$ dominates the probability of jumping to left under $Z_i$ when $Z_i = j \in J_{n_k}$.
As a result, if the process $Z_i$ stays within $J_{n_k}$,
then the random walk $S_i$ will take longer than the process $Z_i$ to reach $\fl{n_k/a_k}$.
That is,
for $k$ sufficiently large (depending on $\w$),
\begin{equation}\label{Nklarge}
P_\omega(\tilde{N}_k> L_k, N_k=\tilde{N}_k)\leq P\left(\inf \left\{i>0:S_i= \left\lceil \frac{n_k}{a_k}\right\rceil \right\}>L_k\right)
\leq  e^{-\d_1 n_k}.
\end{equation}

In order to estimate the decaying rate of the last term in \eqref{main:inequal}, we first find an explicit upper bound of the exponential moment of $\Theta_i$.  Recall that, each $\Theta_i$ is a crossing time from $\nu(Z_{i-1},k)$ to either $\nu(Z_{i-1}-1,k)$ or $\nu(Z_{i-1}+1,k)$ that the walk visits first.  Then, each $\Theta_i$ is less than the crossing time from $\nu(Z_{i-1},k)$ to $\nu(Z_{i-1}+1,k)$ with a reflection at $\nu(Z_{i-1}-1,k)$ for $Z_{i-1}\in \mathbb{Z}$.  Therefore, we have for $\l>0$ that
\begin{align}
E_\omega\left[e^{\lambda \Theta_i\mathbf{I}_{\{Z_{i-1}\in J_{n_k}\}}}\right]&= \sum_{j\in J_{n_k}} P(Z_{i-1}=j)\times E^{\nu(j,k)}_{\omega(\nu(j-1,k))}\left[e^{\lambda T_{\nu(j+1,k)}}\right]+P(Z_{i-1}\notin J_{n_k})\nonumber\\
&\leq\max_{j\in J_{n_k}} E^{\nu(j,k)}_{\omega(\nu(j-1,k))}\left[e^{\lambda T_{\nu(j+1,k)}}\right].\label{thetaine}
\end{align}
By Corollary \ref{cor-1} with $m=\nu(j-1,k), k_0=\nu(j,k)$ and $k_1=\nu(j+1,k)$, the right side of inequality in \eqref{thetaine} has an upper bound in an explicit form. That is, with $\lambda>0$ sufficiently small enough such that
\begin{equation}
\max_{j \in J_{n_k}} E_{\omega(\nu(j-1,k))}^{\nu(j-1,k)}[T_{\nu(j+1,k)}] < \frac{e^{-\l}}{\sinh \l},
\label{condition}
\end{equation}
we have
\[
E^{\nu(j,k)}_{\omega(\nu(j-1,k))}[e^{\l T_{\nu(j+1,k)}}]\leq\exp\left(\frac{\sinh\lambda(E_{\omega(\nu(j-1,k))}^{\nu(j,k)}[T_{\nu(j+1,k)}])}{e^{-\lambda}-\sinh{\lambda}(E_{\omega(\nu(j-1,k))}^{\nu(j-1,k)}[T_{\nu(j+1,k)}])}\right)\] for each $j\in J_{n_k}$.
Therefore, we get
\begin{equation}
E_\omega[e^{\lambda \Theta_i\mathbf{I}_{\{Z_{i-1}\in J_{n_k}\}}}]\leq \max_{j\in J_{n_k}}\exp \left(\frac{\sinh{\lambda}(E_{\omega(\nu(j-1,k))}^{\nu(j,k)}[T_{\nu(j+1,k)}])}{e^{-\lambda}-\sinh{\lambda}(E_{\omega(\nu(j-1,k))}^{\nu(j-1,k)}[T_{\nu(j+1,k)}])}
\right).\label{ref2}
\end{equation}
Note that the requirement that $\lambda>0$ is small enough so that \eqref{condition} is satisfied is needed for \eqref{ref2} to ensure that certain moment generating functions are finite.
Since $e^{-\l}/\sinh \l \ra \infty$ as $\l \ra 0^+$, \eqref{condition} is always satisfied for some small $\l>0$. However, we will later want to apply the upper bound \eqref{ref2} with a deterministic choice of $\l = \l_k = D_0 n_k^{-1/s}$ with some fixed $D_0>0$, and in this case the bound \eqref{condition} may not necessarily be satisfied.
However, we will prove a following claim and show that with this choice of $\l_k$ there is an environment dependent subsequence of $n_k$ where the condition \eqref{condition} is met.
For any fixed constant $\epsilon_1>0$, we will show that
\begin{equation}
Q\left(\max_{j\in J_{n_k}} E_{\omega(\nu(j-1,k))}^{\nu(j-1,k)}[T_{\nu(j+1,k)}]< 2(E_Q[\b_0]+\epsilon_1) a_k \quad i.o \right)=1.\label{ex1}
\end{equation}
Recall that the sequence $a_k = n_k^{1/s}/D$ for some $D>1$.
Since  $\frac{e^{-\l_k}}{\sinh{\l_k}} \sim \frac{1}{\l_k} = \frac{n_k^{1/s}}{D_0}$, it follows from \eqref{ex1} that if the constants $D,D_0$ and $\epsilon_1$ are chosen so that $D>2(E_Q[\b_0]+\epsilon_1)D_0$ then
$$\max_{j\in J_{n_k}} E_{\omega(\nu(j-1,k))}^{\nu(j-1,k)}[T_{\nu(j+1,k)}]<\frac{2(E_Q[\b_0]+\epsilon_1)}{D}n_k^{1/s} \leq \frac{e^{-\l_k}}{\sinh{\l_k}}, \quad \text{infinitely often.} $$
Therefore, it is enough to prove \eqref{ex1} to show that there is almost surely a subsequence of $n_k$ for which \eqref{condition} holds when $\l = \l_k = D_0 n_k^{-1/s}$.

To simplify notation, for any integers $i,j$ such that $i \in [(j-1)a_k,(j+1)a_k-1]$ let $\b_i^j = E_{\w(\nu(j-1,k))}^{\nu_i}[T_{\nu_{i+1}}]$ be the quenched expected crossing time from $\nu_i$ to $\nu_{i+1}$ with a reflection added at $\nu(j-1,k)$.
Then, we can restate \eqref{ex1} as

\begin{equation}
Q\left(\max_{j\in J_{n_k}} \sum_{i=(j-1)a_k}^{(j+1)a_k-1}\beta_i^j<\frac{2(E_Q[\b_0]+\epsilon_1)}{D}n_k^{1/s}\quad i.o\right)=1.\label{Q}
\end{equation}
A strategy for proving \eqref{Q} is to classify the sums of $\beta_i^j$ into two groups by the size of $M_i$ and determine an upper bound of the sums of each group separately. For a fixed $\epsilon > 0$ we will refer to $\{i: \, M_i>n_i^{(1-\epsilon)/s}\}$ and $\{i: \, M_i\leq n_i^{(1-\epsilon)/s}\}$ as ``big hills'' and ``small hills,'' respectively. Then, we begin by a lemma showing the upper bound of a group of $\beta_i^j$ corresponding to small hills using Proposition \ref{prop1}. An upper bound of $\beta_i^j$ corresponding to big hills requires a more careful estimation because $\beta_i^j$ with the biggest hill dominates all of the other $\beta_i^j$'s.  The first step is to prove that $\beta_i^j$ corresponding to big hills are typically located outside of a small group of ladder blocks.  Then, we show that at most one big hill is typically observed at each ladder block.   Finally, we estimate a uniform bound of $\beta_i^j$ corresponding to big hills observed from each ladder block.

The following lemma shows that the maximums of sums of centered expected crossing time with a small hill, $\{M_i\leq n_k^{(1-\epsilon)/s}\}$, are negligible in the limit.
\begin{lem}\label{lm35}
Let us define $J'_{n_k}=J_{n_k}\cup \{-\lfloor n_k/a_k\rfloor-1\}$.  Then, for any $\epsilon_1>0$,
$$Q\left(\max_{j\in J'_{n_k}}\sum_{i=(j)a_k}^{(j+1)a_k-1}(\beta_i^j\mathbf{I}_{\{M_i\leq n_k^{(1-\epsilon)/s}\}}-E_Q[\beta_0])>\frac{\epsilon_1}{2}a_k\quad i.o.\right)=0.$$
\end{lem}
\begin{proof}
Since $\beta_i^j<\beta_i$ for any $j\in J'_{n_k}$ and $i\in[(j)a_k,(j+1)a_k-1]$, it suffices to prove
$$Q\left(\max_{j\in J'_{n_k}}\sum_{i=(j)a_k}^{(j+1)a_k-1}(\beta_i\mathbf{I}_{\{M_i\leq n_k^{(1-\epsilon)/s}\}}-E_Q[\beta_0])>\frac{\epsilon_1}{2}a_k\quad i.o.\right)=0.$$
Recall that $\beta_i,i\in\mathbb{Z}$ is stationary under $Q$.  Hence,
\begin{align*}
&Q\left(\max_{j\in J'_{n_k}}\sum_{i=(j)a_k}^{(j+1)a_k-1}(\beta_i\mathbf{I}_{\{M_i\leq n_k^{(1-\epsilon)/s}\}}-E_Q[\beta_0])>\frac{\epsilon_1}{2}a_k\right)\\
&\quad\quad\leq \frac{3n_k}{a_k}Q\left(\sum_{i=0}^{a_k-1}(\beta_i\mathbf{I}_{\{M_i\leq n_k^{(1-\epsilon)/s}\}}-E_Q[\beta_0])>\frac{\epsilon_1}{2}a_k\right)\\
&\quad\quad\leq C n_k e^{-C'(\log n_k)^2}, \text{ for some }C,\,C'>0,
\end{align*}
where the last equality comes from Proposition \ref{prop1}. Then, the conclusion follows by the Borel-Cantelli Lemma.
\end{proof}

Next, a following lemma shows that the maximum $\beta_i^j$ with big hill always occurs in $j\in J_{n_k}\setminus\{-1,0,1\}$ for $k$ large enough.
\begin{lem}\label{lm33}
If $0<\epsilon<1-1/s$, then
\[
Q\left(\max_{\substack{j\in J_{n_k} \\ i\in[(j-1)a_k,(j+1)a_k-1]}}\beta_i^j\mathbf{I}_{\{M_i>n_k^{(1-\epsilon)/s}\}}\neq
\max_{\substack{j\in J_{n_k}\setminus\{-1,0,1\}\\i\in[(j-1)a_k,(j+1)a_k-1]}}\beta_i^j\mathbf{I}_{\{M_i>n_k^{(1-\epsilon)/s}\}}\quad{}i.o.\right)=0. \]
\end{lem}
\begin{proof}
We have the following inclusion,
\begin{align*}
&\left\{\max_{\substack{j\in J_{n_k} \\ i\in[(j-1)a_k,(j+1)a_k-1]}}\beta_i^j\mathbf{I}_{\{M_i>n_k^{(1-\epsilon)/s}\}}\neq\max_{\substack{j\in J_{n_k}\setminus\{-1,0,1\} \\ i\in[(j-1)a_k,(j+1)a_k-1]}}\beta_i^j\mathbf{I}_{\{M_k>n_k^{(1-\epsilon)/s}\}}\right\}\\
&\quad=\left\{\max_{\substack{j\in \{-1,0,1\} \\ i\in[(j-1)a_k,(j+1)a_k-1]}}\beta_i^j\mathbf{I}_{\{M_i>n_k^{(1-\epsilon)/s}\}}>\max_{\substack{j\in J_{n_k}\setminus\{-1,0,1\} \\ i\in[(j-1)a_k,(j+1)a_k-1]}}\beta_i^j\mathbf{I}_{\{M_k>n_k^{(1-\epsilon)/s}\}}\right\}\\
&\quad\subset\{\max_{-2a_k\leq i\leq 2a_k-1}M_i>n_k^{(1-\epsilon)/s}\}.
\end{align*}
That is, in order for the two maximums to not be equal there must be at least one large hill corresponding to some $i \in [-2a_k,2a_k-1]$.
Moreover, for large $n_k$,
$$
Q\left(\max_{-2a_k\leq i\leq 2a_k-1}M_i>n_k^{(1-\epsilon)/s}\right)=(4a_k)Q(M_0> n_k^{(1-\epsilon)/s})
= O\left( \frac{a_k}{n_k^{1-\epsilon}} \right) = O\left(\frac{1}{n_k^{1-\epsilon-1/s}}\right),
$$
where the second to last equality comes from the tail asymptotics of $M_0$ in \eqref{M} and the last equality comes from the definition of $a_k$.  Then, the conclusion of the lemma follows from the Borel-Cantelli Lemma.
\end{proof}
A following lemma shows that for $n_k$ large enough each interval  $[(i-1)a_k,(i+1)a_k-1]$ with $i \in J_{n_k}$ contains at most one big hill.
\begin{lem}\label{lm34}
If $0<\epsilon<\frac{s-1}{2s}$, then
$$Q\left(\exists j\in J_{n_k}\text{ such that }\sharp \{i\in[(j-1)a_k,(j+1)a_k-1]:M_i>n_k^{(1-\epsilon)/s}\}\geq2\text{ for infinitely many k}\right)=0.$$
\end{lem}
\begin{proof}
Since $\{M_i\}_{i\in\mathbb{Z}}$ is i.i.d. under $Q$,
\begin{align*}
&Q \left( \exists j\in J_{n_k}\text{ such that }\sharp \{i\in[(j-1)a_k,(j+1)a_k-1]:M_i>n_k^{(1-\epsilon)/s}\}\geq 2 \right)\\
&\qquad \leq3\frac{n_k}{a_k}Q(\sharp \{i\in[0,2a_k-1]:M_i>n_k^{(1-\epsilon)/s}\}\geq 2).
\end{align*}
For simplicity, let us denote $N:=\sharp \{i\in[0,2a_k-1]:M_i>n_k^{(1-\epsilon)/s}\}$.  Then, $N$ is a binomial random variable with parameter $n=2a_k$ and $p=Q(M_0>n_k^{(1-\epsilon)/s})$.  Using the inequality $(1-np)\leq(1-p)^n$ for $n\geq0$ and $0\leq p\leq1$,
\begin{align*}
Q(N\geq 2)&=1-(1-p)^n-np(1-p)^{n-1}\leq n(n-1)p^2\leq (np)^2.
\end{align*}
Recall that $a_k = n_k^{1/s}/D$ with some fixed constant $D>1$ and $Q(M_0>n_k^{(1-\epsilon)/s})\leq Cn^{\epsilon-1}$ for some constant $C>0$.  Then, we have
\[
3\frac{n_k}{a_k}P(N\geq2)\leq 3\frac{n_k}{a_k}\left(\frac{2Ca_k}{n_k^{1-\epsilon}}\right)^2\leq\frac{C'}{n_k^{1-1/s-2\epsilon}},\quad \text{for some }C'>0.
\]
Since $1-1/s-2\epsilon>0$ by our assumption, the conclusion follows from the Borel-Cantelli Lemma.
\end{proof}

Finally, we show that for some subsequence of $n_k$ the sums of $\beta_i^j$ corresponding to big hills are bounded above by $\epsilon' n_k^{1/s}$ for any $\epsilon'>0$.
\begin{cor}\label{lm36}
Suppose $0<\epsilon<\frac{s-1}{2s}$.  Then, for any $\epsilon'>0$,
$$Q\left(\max_{j\in J_{n_k}} \sum_{i=(j-1)a_k}^{(j+1)a_k-1}\beta_i^{j}\mathbf{I}_{\{M_i> n^{(1-\epsilon)/s}\}}<\epsilon' n_k^{1/s}\quad{}i.o.\right)=1$$
\end{cor}
\begin{proof}
First, we prove that,
\begin{equation}
Q\left(\max_{ \substack{j\in J_{n_k}\setminus\{-1,0,1\} \\ i\in[(j-1)a_k,(j+1)a_k-1]}  }  \beta_i^j\mathbf{I}_{\{M_i> n_k^{(1-\epsilon)/s}\}}<\epsilon' n_k^{1/s}\quad i.o.\right)=1. \label{final2}
\end{equation}
Since $n_k=m^{m^k}$ for some $m>s$ and $a_k = n_k^{1/s}/D$ we have that $n_{k-1} < a_k$ for all $k$ large enough. Therefore, $[\nu_{-n_{k-1}},\nu_{n_{k-1}}] \subset [\nu_{-a_k},\nu_{a_k}] = [\nu(-1,k),\nu(1,k)]$ and due to the reflections used in the definition of $\b_i^j$ the event inside the probability in \eqref{final2} is independent of the environment in the interval $[\nu_{-n_{k-1}},\nu_{n_{k-1}}]$.
Therefore, the events inside \eqref{final2} are an independent sequence for $k$ large enough and so to prove \eqref{final2} by the second Borel-Cantelli Lemma it is enough to show that
\begin{equation}
\sum_{k=1}^\infty Q\left(
\max_{ \substack{j\in J_{n_k}\setminus\{-1,0,1\} \\ i\in[(j-1)a_k,(j+1)a_k-1]}  }
\beta_i^j\mathbf{I}_{\{M_i> n_k^{(1-\epsilon)/s}\}}<\epsilon' n_k^{1/s}\right)=\infty. \label{2bc}
\end{equation}
To prove \eqref{2bc}, note that
\begin{align}
Q\left(\max_{ \substack{j\in J_{n_k}\setminus\{-1,0,1\} \\ i\in[(j-1)a_k,(j+1)a_k-1]}  }  \beta_i^j\mathbf{I}_{\{M_i> n_k^{(1-\epsilon)/s}\}}<\epsilon' n_k^{1/s}\right)
&\geq Q\left(\max_{ \substack{j\in J_{n_k} \\ i\in[(j-1)a_k,(j+1)a_k-1]}} \beta_i^j\mathbf{I}_{\{M_i> n_k^{(1-\epsilon)/s}\}}<\epsilon' n_k^{1/s}\right)\nonumber\\
&\geq Q\left(\max_{ \substack{j\in J_{n_k} \\ i\in[(j-1)a_k,(j+1)a_k-1]}} \frac{\beta_i^j}{(2n_k)^{1/s}}<\frac{\epsilon'}{2}\right)\nonumber\\
&\geq Q\left(\max_{i\in[-n_k,n_k-1]} \frac{\beta_i}{(2n_k)^{1/s}}<\frac{\epsilon'}{2}\right). \label{maxbeta}
\end{align}
It was shown in \cite[Proposition 5.1]{PS13} that
$\{\frac{\beta_i}{(2n)^{1/s}},-n\leq i <n\}$ converges weakly to a nonhomogeneous Poisson point process with intensity measure $\gamma x^{-s-1}dx$ for some $\gamma>0$.
Hence, the probabilities in \eqref{maxbeta} are uniformly bounded away from 0 for all $k$ and thus \eqref{2bc} follows.

By Lemma \ref{lm33} and \ref{lm34}, we have, for $k$ large enough,
\begin{align}
\max_{\substack{j\in J_{n_k}\setminus\{-1,0,1\}\\i\in[(j-1)a_k,(j+1)a_k-1]}} \beta_i^j\mathbf{I}_{\{M_i> n_k^{(1-\epsilon)/s}\}}&=\max_{\substack{j\in J_{n_k}\\i\in[(j-1)a_k,(j+1)a_k-1]}} \beta_i^j\mathbf{I}_{\{M_i> n_k^{(1-\epsilon)/s}\}} \nonumber \\
&=\max_{j\in J_{n_k}}\sum_{i=(j-1)a_k}^{(j+1)a_k-1} \beta_i^j\mathbf{I}_{\{M_i> n_k^{(1-\epsilon)/s}\}}. \label{3betamax}
\end{align}
Hence, the conclusion of the Corollary follows from \eqref{final2} and \eqref{3betamax}.
\end{proof}

We are now ready to give the proof of the main result of this section.
\begin{proof}[\textit{Proof of Proposition }\ref{prop2}]
Recall that $a_k=\frac{n_k^{1/s}}{D}$ and $L_k=\frac{n_k}{a_k(1-\delta)} = \frac{D}{1-\delta} n_k^{1-1/s}$ for some fixed $\delta>0$.  And, choose $\lambda=\lambda_k=\frac{D_0}{n_k^{1/s}}$ for any fixed $D_0>0.$
Recall from \eqref{main:inequal} that
$$P_\omega( T_{\nu_{n_k}} > u \nu_{n_k} ) \leq
 P_\omega(N_k \neq \tilde{N}_k) +  P_\omega(  \tilde{N}_k  > L_k, \,  N_k = \tilde{N}_k  ) + P_\omega\left( \sum_{i=1}^{\tilde{N}_k} \Theta_i > u n_k, \, \tilde{N}_k \leq L_k \right).$$
We have proved in Lemma \ref{main:lemma1} and \eqref{Nklarge} that the first two terms on the right side decay exponentially for $Q$-a.e.\ environment $\w$.  Consequently,
\begin{equation}
\lim_{n\to\infty}\frac{1}{n^{1-1/s}}\log \left\{P_\omega(N_k \neq \tilde{N}_k)+P_\omega( \tilde{N}_k  > L_k, \, N_k = \tilde{N}_k )\right\}=-\infty.\label{PA1A2}
\end{equation}
Regarding the third term, for each $i \leq \tilde{N}_k$ the distribution of the crossing time $\Theta_i$ is determined by the location $Z_{i-1}\in J_{n_k}$.  Also, since $\tilde{N}_k\leq L_k$,
\begin{align}
P_\omega\left( \sum_{i=1}^{\tilde{N}_k} \Theta_i > u \nu_{n_k}, \, \tilde{N}_k \leq L_k \right)&\leq P_\omega\left(\sum_{i=1}^{L_k}\Theta_i\mathbf{I}_{\{Z_{i-1}\in J_{n_k}\}}>u\nu_{n_k} \right)\nonumber\\
&\leq E_\omega\left[\prod_{i=1}^{L_k}e^{\lambda_k\Theta_i\mathbf{I}_{\{Z_{i-1}\in J_{n_k}\}}}\right]e^{-\lambda_k u\nu_{n_k}},\label{A3}
\end{align}
where the second inequality comes from Chebyshev's inequality.
Now, we claim that
\begin{equation}
 E_\omega\left[\prod_{i=1}^{L}e^{\lambda_k\Theta_i\mathbf{I}_{\{Z_{i-1}\in J_{n_k}\}}}\right]
\leq \left(\max_{j\in{J_{n_k}}}E^{\nu(j,k)}_{\omega(\nu(j-1,k))}[e^{\lambda_k T_{\nu(j+1,k)}}]\right)^{L}, \quad \text{for any } L \geq 1. \label{thetamgfub}
\end{equation}
To see this, let $\mathfrak{G}_{i}:=\sigma(X_n:n\leq\sum_{l=1}^i\Theta_l)$ be the $\sigma$-field generated by the walk up until the $i$-th step of the induced birth-death process on the blocks.
Then,
\begin{align}
E_\omega\left[\prod_{i=1}^{L}e^{\lambda_k\Theta_i\mathbf{I}_{\{Z_{i-1}\in J_{n_k}\}}}\right]&= E_\omega\left[E_\omega\left[\prod_{i=1}^{L}e^{\lambda_k\Theta_i\mathbf{I}_{\{Z_{i-1}\in J_{n_k}\}}}|\mathfrak{G}_{L-1}\right]\right]\nonumber\\ &=E_\omega\left[\prod_{i=1}^{L-1}e^{\lambda_k\Theta_i\mathbf{I}_{\{Z_{i-1}\in J_{n_k}\}}}E_\omega\left[e^{\lambda_k\Theta_{L}\mathbf{I}_{\{Z_{L-1}\in J_{n_k}\}}}|\mathfrak{G}_{L-1}\right]\right]\nonumber\\
&\leq \max_{j\in{J_{n_k}}}E^{\nu(j,k)}_{\omega(\nu(j-1,k))}[e^{\lambda_k T_{\nu(j+1,k)}}]\times E_\omega\left[\prod_{i=1}^{L-1}e^{\lambda_k\Theta_i\mathbf{I}_{\{Z_{i-1}\in J_{n_k}\}}}\right],\nonumber
\end{align}
where the last inequality comes from \eqref{thetaine}, and then \eqref{thetamgfub} follows by induction.
Applying \eqref{thetamgfub} with $L=L_k$, we have
\begin{align}
&\liminf_{k\to\infty}\frac{1}{n_{k}^{1-1/s}}\log P_\omega\left( \sum_{i=1}^{\tilde{N}_k} \Theta_i > u \nu_{n_k}, \, \tilde{N}_k \leq L_k \right) \nonumber \\
&\qquad\leq\liminf_{k\to\infty}\frac{1}{n_{k}^{1-1/s}}\log \left\{ \left(\max_{j\in{J_{n_{k}}}}E^{\nu(j,k)}_{\omega(\nu(j-1,k))}[e^{\lambda_k T_{\nu(j+1,k)}}]\right)^{L_{k}}e^{-\lambda_k u\nu_{n_{k}}} \right\} \nonumber\\
&\qquad=\liminf_{k\to\infty}\frac{L_{k}}{n_{k}^{1-1/s}}\left(\log\max_{j\in J_{n_{k}}}E^{\nu(j,k)}_{\omega(\nu(j-1,k))}[e^{\lambda_k T_{\nu(j+1,k)}}]\right)-\frac{\lambda_k u\nu_{n_{k}}}{n_{k}^{1-1/s}}\nonumber\\
&\qquad\leq\liminf_{k\to\infty}\frac{L_{k}}{n_{k}^{1-1/s}}\left(\max_{j\in J_{n_{k}}}\frac{\sinh{\lambda_k}(E_{\omega(\nu(j-1,k))}^{\nu(j,k)}[T_{\nu(j+1,k)}])}{e^{-\lambda_k}-\sinh{\lambda_k}(E_{\omega(\nu(j-1,k))}^{\nu(j-1,k)}[T_{\nu(j+1,k)}])}\right)-\frac{\lambda_k u\nu_{n_{k}}}{n_{k}^{1-1/s}},\label{conc}
\end{align}
where the first inequality comes from \eqref{A3} and \eqref{thetamgfub}, and the last inequality comes from \eqref{ref2}.
Recall from Lemma \ref{lm35} that, for any $\epsilon_1>0$ and $0<\epsilon<\frac{s-1}{2s}$, there is a $K(\omega)$ such that for all $k\geq K(\omega)$,
\begin{equation}
\max_{j\in J_{n_k}}\sum_{i=(j)a_k}^{(j+1)a_k-1} \beta_i^j\mathbf{I}_{\{M_i\leq n_k^{(1-\epsilon)/s}\}}\leq \frac{E_Q[\b_0]+\epsilon_1/2}{D}n_k^{1/s}.\label{ineq1}
\end{equation}
On the other hand, by Corollary \ref{lm36} with $\epsilon'=\epsilon_1/D$, we can find an environment dependent subsequence of $n_k$ defined as $n_{k'}$ such that
\begin{equation}
\max_{j\in J_{n_{k'}}} \sum_{i=(j-1)a_{k'}}^{(j+1)a_{k'}-1}\beta_i^{j}\mathbf{I}_{\{M_i> n_{k'}^{(1-\epsilon)/s}\}}<\frac{\epsilon_1}{D}n_{k'}^{1/s}. \label{ineq2}
\end{equation}
Then, by a choice of $D>2(E_Q[\b_0]+\epsilon_1)D_0$, \eqref{ex1} is satisfied for some subsequence $k'$.
Therefore, we can conclude that \eqref{conc} is bounded above by
\begin{align}
& \lim_{k \to\infty}\frac{L_{k}}{n_{k}^{1-1/s}}\frac{\sinh(\l_{k}) \frac{E_Q[\b_0]+3\epsilon_1/2}{D} n_{k}^{1/s}}{e^{-\l_{k}}-2 \sinh(\l_{k}) \frac{E_Q[\b_0]+\epsilon_1}{D} n_{k}^{1/s}}-\frac{\lambda_{k} u \, \nu_{n_{k}}}{n_{k}^{1-1/s}} \nonumber \\
&\quad\quad= \frac{D_0(E_Q[\b_0]+3\epsilon_1/2)}{(1-\delta)(1-2D_0(E_Q[\b_0]+\epsilon_1)/C)}-D_0uE_Q[\nu_1],\qquad Q\text{-a.s.},\label{conclu0}
\end{align}
where in the last equality we used that $\l_k = D_0 n_k^{-1/s}$, $L_k = \frac{D}{1-\d} n_k^{1-1/s}$ and the fact that $\nu_n/n \rightarrow E_Q[\nu_1]$, $Q$-a.s.
In summary, we have shown that for any $D_0,\epsilon_1,\d>0$ and for all sufficiently large $D<\infty$ that
\[
 \liminf_{n\ra\infty} \frac{1}{n^{1-1/s}} \log P_\w( T_{\nu_n} > u{\nu_n} ) \leq \frac{D_0(E_Q[\b_0]+3\epsilon_1/2)}{(1-\delta)(1-2D_0(E_Q[\b_0]+\epsilon_1)/D)}-D_0uE_Q[\nu_1],\qquad Q\text{-a.s.}.
\]
By first taking $D\rightarrow\infty$ and then letting $\epsilon_1,\d\rightarrow 0$, we can thus conclude that
\begin{equation}\label{optlim}
 \liminf_{n\ra\infty} \frac{1}{n^{1-1/s}} \log P_\w( T_{\nu_n} > u{\nu_n} ) \leq D_0E_Q[\nu_1]\left(\frac{E_Q[\b_0]}{E_Q[\nu_1]}-u\right),\qquad Q\text{-a.s.},
\end{equation}
for any $D_0 < \infty$.
Finally, since
\[
 \frac{1}{v_\alpha}
= \lim_{n\to\infty}\frac{T_n}{n}
=\lim_{n\to\infty}\frac{T_{\nu_n}}{\nu_n}
=\lim_{n\to\infty}\frac{T_{\nu_n}}{n}\frac{n}{\nu_n}
= \frac{E_Q[ \b_0] }{E_Q[\nu_1]},
\]
it follows that the term in parenthesis in \eqref{optlim} is negative for $u > 1/v_\a$, and thus the right side of \eqref{optlim} can be made smaller than any negative number by choosing $D_0$ sufficiently large.
\end{proof}

\bibliographystyle{alpha}
\bibliography{Reference_Advanced}
\end{document}